\documentclass[10pt,reqno]{amsart}
\usepackage{amsmath,amscd,amsfonts,amsthm,amssymb,latexsym,mathrsfs,graphicx,url}
\usepackage{todonotes, pgfplots, tikz, subcaption, lipsum, epstopdf, enumitem }
\usepackage{tikz-cd}
\usetikzlibrary{shapes.geometric}
\definecolor{myCiteColour}{rgb}{0,0,0}
\definecolor{myLinkColour}{rgb}{0,0,0}

\usepackage{hyperref}
 \hypersetup{
     bookmarks=true,         % show bookmarks bar?
     unicode=false,          % non-Latin characters in AcrobatÕs bookmarks
     pdftoolbar=true,        % show AcrobatÕs toolbar?
     pdfmenubar=true,        % show AcrobatÕs menu?
     pdffitwindow=false,     % window fit to page when opened
     pdfstartview={FitH},    % fits the width of the page to the window
     pdftitle={My title},    % title
     pdfauthor={Author},     % author
     pdfsubject={Subject},   % subject of the document
     pdfcreator={Creator},   % creator of the document
     pdfproducer={Producer}, % producer of the document
     pdfkeywords={keyword1} {key2} {key3}, % list of keywords
     pdfnewwindow=true,      % links in new window
     colorlinks=true,       % false: boxed links; true: colored links
     linkcolor=myLinkColour,          % color of internal links
     citecolor=myCiteColour,        % color of links to bibliography
     filecolor=blue,      % color of file links
     urlcolor=blue      % color of external links
    }

\theoremstyle{plain}
   \newtheorem{theorem}{Theorem}[section]
   \newtheorem{proposition}[theorem]{Proposition}
   \newtheorem{lemma}[theorem]{Lemma}
   \newtheorem{corollary}[theorem]{Corollary}
   
   \newtheorem{conjecture}[theorem]{Conjecture}
   
\theoremstyle{definition}
   \newtheorem{definition}{Definition}[section]
   \newtheorem{example}{Example}[section]

\theoremstyle{remark}
   \newtheorem{remark}[theorem]{Remark}

\numberwithin{equation}{section}

\def\kk{\kern.2ex\mbox{\raise.5ex\hbox{{\rule{.35em}{.12ex}}}}\kern.2ex}

\newcommand{\lma}{\lambda_{\rm max}}
\newcommand{\lmi}{\lambda_{\rm min}}

\newcommand{\zz}{\mathbf{z}}
\newcommand{\ww}{\mathbf{w}}
\newcommand{\ee}{\mathbf{e}}

\newcommand{\xx}{\mathbf{x}}
\newcommand{\yy}{\mathbf{y}}
\newcommand{\uu}{\mathbf{u}}
\renewcommand{\ss}{\mathbf{s}}
\newcommand{\vv}{\mathbf{v}}

\newcommand{\bl}{\boldsymbol{\lambda}}

\newcommand{\RR}{\mathbb{R}}

\renewcommand{\Im}{{\rm Im}}

\DeclareMathOperator{\conv}{\ast}

\allowdisplaybreaks[1]

\title[]{Spectrahedrality of hyperbolicity cones of multivariate matching polynomials} 

\author{Nima Amini}

\address{Department of Mathematics, Royal Institute of Technology, SE-100 44 Stockholm,
Sweden}
\email{namini@kth.se}

\begin{document}
\begin{abstract}
The generalized Lax conjecture asserts that each hyperbolicity cone is a linear slice of the cone of positive semidefinite matrices. We prove the conjecture for a multivariate generalization of the matching polynomial. This is further extended (albeit in a weaker sense) to a multivariate version of the independence polynomial for simplicial graphs. As an application we give a new proof of the conjecture for elementary symmetric polynomials (originally due to Br\"and\'en). Finally we consider a hyperbolic convolution of determinant polynomials generalizing an identity of Godsil and Gutman. 
\end{abstract}
\maketitle

\renewcommand\theenumi{\roman{enumi}}
\thispagestyle{empty}

\section{Introduction} \label{sec::intro}
A homogeneous polynomial $h(\xx) \in \RR[x_1, \ldots, x_n]$ is \emph{hyperbolic} with respect to a vector 
$\ee \in \RR^n$ if $h(\ee) \neq 0$, and if for all $\xx \in \RR^n$ the univariate polynomial $t \mapsto h(t\ee -\xx)$ has only real zeros. Note that if $h$ is a hyperbolic polynomial of degree $d$, then we may write
\begin{align*}
\displaystyle h(t\mathbf{e} - \mathbf{x}) = h(\mathbf{e}) \prod_{j=1}^d (t - \lambda_j(\mathbf{x})), 
\end{align*} \noindent
where $$\lma(\xx)=\lambda_1(\mathbf{x}) \geq \cdots \geq \lambda_d(\mathbf{x})=\lmi(\xx)$$ are called the \emph{eigenvalues} of $\mathbf{x}$ (with respect to $\mathbf{e}$). The \emph{hyperbolicity cone} of $h$ with respect to $\mathbf{e}$ is the set $\Lambda_{+}(h, \mathbf{e}) = \{ \mathbf{x} \in \mathbb{R}^n : \lmi(\mathbf{x}) \geq 0 \}$. If $\vv \in \Lambda_{+}(h, \ee)$, then $h$ is hyperbolic with respect to $\vv$ and $\Lambda_{+}(h,\vv) = \Lambda_{+}(h, \ee)$. For this reason we usually abbreviate and write $\Lambda_{+}(h)$ if there is no risk for confusion. We denote by $\Lambda_{++}(h)$ the interior of $\Lambda_{+}(h)$. The cone $\Lambda_{++}(h)$ is convex and can be characterized as the connected component of the set $\{\xx \in \RR^n :h(\xx) \neq 0 \}$ containing $\ee$. These are all facts due to G\r{a}rding \cite{Gar}.
\begin{example}\label{detta}
An important example of a hyperbolic polynomial is $\det(X)$, where $X = (x_{ij})_{i,j =1}^n$ is a matrix of variables where we impose $x_{ij} = x_{ji}$. Note that $t \mapsto \det(tI - X)$ where $I = \text{diag}(1,\dots, 1)$, is the characteristic polynomial of a symmetric matrix so it has only real zeros. Hence $\det(X)$ is a hyperbolic polynomial with respect to $I$, and its hyperbolicity cone is the cone of positive semidefinite matrices. 
\end{example} \noindent
Denote the directional derivative of $h(\xx) \in \RR[x_1, \dots, x_n]$ with respect to $\vv = (v_1, \dots, v_n)^T \in \RR^n$ by
$$
\displaystyle D_{\vv}h(\xx) = \sum_{k=1}^n v_k \frac{\partial h}{\partial x_k} (\xx).
$$ 
The following lemma is well-known and essentially follows from the identity $D_{\vv}h(t) = \frac{d}{dt}h(t \vv + \xx) |_{t=0}$ together with Rolle's theorem (see \cite{Gar} \cite{R}).
\begin{lemma} \label{derivhyp}
Let $h$ be a hyperbolic polynomial and let $\vv \in \Lambda_+$ be such that $D_{\vv}h \not \equiv 0$. Then $D_{\vv} h$ is hyperbolic with $\Lambda_{+}(h, \vv) \subseteq \Lambda_{+}(D_{\vv} h, \vv)$.
\end{lemma} \noindent
A class of polynomials which is intimately connected to hyperbolic polynomials is the class of stable polynomials.
A polynomial $P(\xx) \in \mathbb{C}[x_1, \dots, x_n]$ is \emph{stable} if $P(z_1,\ldots,z_n) \neq 0$ whenever $\Im(z_j)>0$ for all $1\leq j \leq n$.
Hyperbolic and stable polynomials are related as follows, see \cite[Prop. 1.1]{PLMS}. 
\begin{lemma}
\label{hypbas}
Let $P \in \mathbb{R}[x_1, \dots, x_n]$ be a homogenous polynomial. Then $P$ is stable if and only if $P$ is hyperbolic with $\mathbb{R}_+^n \subseteq \Lambda_+(P)$.
\end{lemma} \noindent
The next theorem which follows (see \cite{LPR}) from a theorem of Helton and Vinnikov \cite{HV} proved the Lax conjecture (after Peter Lax $1958$ \cite{Lax}). 
\begin{theorem}[Helton-Vinnikov \cite{HV}] \label{heltonvinnikov}
Suppose that $h(x,y,z)$ is of degree $d$ and hyperbolic with respect to $e = (e_1,e_2,e_3)^T$. Suppose further that $h$ is normalized such that $h(e) = 1$. Then there are symmetric $d \times d$ matrices $A, B, C$ such that $e_1A+e_2B+ e_3C = I$ and
\begin{align*}
\displaystyle h(x,y,z) = \det(xA + yB + zC).
\end{align*}
\end{theorem} \noindent
\begin{remark}
The exact analogue of Theorem \ref{heltonvinnikov} fails for $n>3$ variables. This may be seen by comparing dimensions. The set of polynomials on $\mathbb{R}^n$ of the form $\det(x_1A_1 + \cdots x_nA_n)$ with $A_i$ symmetric for $1 \leq i \leq n$, has dimension at most $n \binom{d}{2}$ whereas the set of hyperbolic polynomials on $\mathbb{R}^n$ has dimension $\binom{n+d-1}{d}$.   
\end{remark} \noindent
A convex cone in $\RR^n$ is \emph{spectrahedral} if it is of the form
\begin{align*}
\displaystyle \left \{ \mathbf{x} \in \mathbb{R}^n : \sum_{i = 1}^n x_i A_i \thickspace \text{ is positive semidefinite} \right \}
\end{align*} \noindent
where $A_i$, $i = 1, \dots, n$ are symmetric matrices such that there exists a vector $(y_1, \dots, y_n) \in \mathbb{R}^n$ with $\sum_{i=1}^n y_i A_i$ positive definite. It is easy to see that spectrahedral cones are hyperbolicity cones. A major open question asks if the converse is true. 
\begin{conjecture}[Generalized Lax conjecture \cite{HV,Vin}] \label{glc}
All hyperbolicity cones are spectrahedral.
\end{conjecture} \noindent
\begin{remark}
An important consequence of Conjecture \ref{glc} in the field of optimization is that hyperbolic programming \cite {R} is the same as semidefinite programming.
\end{remark} \noindent
We may reformulate Conjecture \ref{glc} as follows, see \cite{HV,Vin}. The hyperbolicity cone of $h(\mathbf{x})$ with respect to $\mathbf{e} = (e_1, \dots ,e_n)$ is spectrahedral if there is a homogeneous polynomial $q(\mathbf{x})$ and real symmetric matrices $A_1, \ldots, A_n$ of the same size such that
\begin{align}\label{qhd}
\displaystyle q(\mathbf{x})h(\mathbf{x}) = \det \left ( \sum_{i = 1}^n x_i A_i \right )
\end{align} \noindent
where $\Lambda_{++}(h, \mathbf{e}) \subseteq \Lambda_{++}(q, \mathbf{e})$ and $\sum_{i=1}^n e_iA_i$ is positive definite. 
If we can choose $q(\xx) \equiv 1$, then we say that $h(\xx)$ admits a \textit{definite determinantal representation}.
\begin{itemize}
\item Conjecture \ref{glc} is true for $n=3$ by Theorem \ref{heltonvinnikov},
\item Conjecture \ref{glc} is true for homogeneous cones \cite{Chua}, i.e., cones for which the automorphism group acts transitively on its interior, 
\item Conjecture \ref{glc} is true for quadratic polynomials, see e.g. \cite{NT}, 
\item Conjecture \ref{glc} is true for elementary symmetric polynomials, see \cite{B2},
\item Weaker versions of Conjecture \ref{glc} are true for smooth hyperbolic polynomials, see \cite{Kum2, NS}. 
\item Stronger algebraic versions of Conjecture \ref{glc} are false, see \cite{AB, B}. 
\end{itemize} \noindent
The paper is organized as follows. In Section \ref{sec::matching} we prove Conjecture \ref{glc} for a multivariate generalization of the matching polynomial (Theorem \ref{matchconespec}). We also show that this implies Conjecture \ref{glc} for elementary symmetric polynomials (Theorem \ref{elemspec}). Our result may therefore be viewed as a generalization of \cite{B2}. In Section \ref{sec::indep} we generalize further to a multivariate version of the independence polynomial using a recent divisibility relation of Leake and Ryder \cite{LR} (Theorem \ref{indepspec}). The variables of the homogenized independence polynomial do not fully correspond combinatorially (under the line graph operation) to the more refined homogeneous matching polynomial. The restriction of Theorem \ref{indepspec} to line graphs is therefore weaker than Theorem \ref{matchconespec}.  
Finally, in Section \ref{sec::conv} we consider a hyperbolic convolution of determinant polynomials generalizing an identity of Godsil and Gutman \cite{GG} which asserts that the expected characteristic polynomial of a random signing of the adjacency matrix of a graph is equal to its matching polynomial.

Unless stated otherwise, $G = (V(G),E(G))$ denotes a simple undirected graph.
We shall adopt the following notational conventions.
\begin{itemize}
\item $\text{Sym}(S)$ denotes the symmetric group on the set $S$. Write $\mathfrak{S}_n = \text{Sym}([n])$.
\item $N_G[v]$ (resp. $N_G(v)$) denotes the closed (resp. open) neighbourhood of $v$.
\item If $S \subseteq V(G)$, then $G[S]$ denotes the subgraph of $G$ induced by $S$.
\item $G \sqcup H$ denotes the disjoint union of the graphs $G$ and $H$.
\item $\RR^S = \{ (a_s)_{s \in S} : a_s \in \RR \} \cong \RR^{|S|}$.
\item $\RR^G = \RR^{V(G)} \times \RR^{E(G)}$.
\end{itemize}

\section{Hyperbolicity cones of multivariate matching polynomials} \label{sec::matching} \noindent
A \textit{$k$-matching} in $G$ is a subset $M \subseteq E(G)$ of $k$ edges, no two of which have a vertex in common. Let $\mathcal{M}(G)$ denote the set of all matchings in $G$ and let $m(G,k)$ denote the number of $k$-matchings in $G$. By convention $m(G,0) = 1$. We denote by $V(M)$ the set of vertices contained in the matching $M$.
If $|V(M)| = |V(G)|$, then we call $M$ a \textit{perfect matching}. The (univariate) \textit{matching polynomial} is defined by
\begin{align*}
\displaystyle \mu(G, t)= \sum_{k \geq 0} (-1)^k m(G,k) t^{|V(G)|-2k}.
\end{align*} \noindent
Note that this is indeed a polynomial since $m(G,k) = 0$ for $k > \frac{|V(G)|}{2}$. Heilmann and Lieb \cite{HL} studied the following multivariate version of the matching polynomial with variables $\mathbf{x} = (x_i)_{i \in V}$ and non-negative weights $\boldsymbol{\lambda} = (\lambda_e)_{e \in E}$,
\begin{align*}
\displaystyle \mu_{\boldsymbol{\lambda}}(G, \mathbf{x}) = \sum_{M \in \mathcal{M}(G)} (-1)^{|M|} \prod_{ij \in M} \lambda_{ij} x_i x_j.
\end{align*} \noindent 

\begin{remark} \label{specrem}
Note that $\displaystyle t^{|V(G)|} \mu_{\mathbf{1}}(G, t^{-1} \mathbf{1}) = \mu(G,t)$, where $\mathbf{1} = (1, \dots, 1)$.
\end{remark} \noindent
\begin{theorem}[Heilmann-Lieb \cite{HL}]
\label{heilmannlieb} \noindent \newline
If $\boldsymbol{\lambda} = (\lambda_e)_{e \in E}$ is a sequence of non-negative edge weights, then $\mu_{\boldsymbol{\lambda}}(G, \mathbf{x})$ is stable.
\end{theorem} \noindent
\begin{remark}
A quick way to see Theorem \ref{heilmannlieb} is to observe that
\begin{align*}
\displaystyle \text{MAP} \left ( \prod_{e = (i,j) \in E(G)} (1- \lambda_{e} x_i x_j) \right ) = \mu_{\boldsymbol{\lambda}}(G, \mathbf{x})
\end{align*} \noindent
where $\text{MAP}: \mathbb{C}[z_1, \dots, z_n] \to \mathbb{C}[z_1, \dots, z_n]$ is the stability preserving linear map taking a multivariate polynomial to its multiaffine part (see \cite{BB}). Since real stable univariate polynomials are real-rooted the Heilmann-Lieb theorem (together with Remark \ref{specrem}) implies the real-rootedness of $\mu(G, t)$.
\end{remark} \noindent
We will consider the following homogeneous multivariate version of the matching polynomial
\begin{definition}
Let $\mathbf{x} = (x_v)_{v \in V}$ and $\mathbf{w} = (w_e)_{e \in  E}$ be indeterminates.
Define the \textit{homogeneous multivariate matching polynomial} $\mu(G, \mathbf{x} \oplus \mathbf{w} ) \in \mathbb{R}[\mathbf{x}, \mathbf{w}]$ by
\begin{align*}
\displaystyle \mu(G, \mathbf{x} \oplus \mathbf{w} ) = \sum_{M \in \mathcal{M}(G)} (-1)^{|M|} \prod_{v \not \in V(M)} x_v \prod_{e \in M} w_e^2.
\end{align*} \noindent
\end{definition} \noindent

\begin{figure}
    \centering
\begin{tikzpicture}

\node[label=right:{ \scriptsize $1$}, draw=black,fill,circle,inner sep=0pt,minimum size=3pt] (x1) at (2,0) {};

\node[label=left:{\scriptsize $4$}, draw=black,fill,circle,inner sep=0pt,minimum size=3pt] (x4) at (-1,0) {};

\node[label=below left:{\scriptsize $3$}, draw=black,fill,circle,inner sep=0pt,minimum size=3pt] (x3) at (-2,-2) {};

\node[label=below right:{\scriptsize $2$}, draw=black,fill,circle,inner sep=0pt,minimum size=3pt] (x2) at (1,-2) {};

\node (e) at (0.3, 0.2) {\scriptsize $e$};
\node (d) at (-1.8, -1) {\scriptsize $d$};
\node (b) at (-0.5, -2.3) {\scriptsize $b$};
\node (a) at (1.7, -1) {\scriptsize $a$};
\node (c) at (0.2, -1.2) {\scriptsize $c$};

\draw (x4) -- (x1);
\draw (x4) -- (x3);
\draw (x3) -- (x2);
\draw (x2) -- (x1);
\draw (x3) -- (x1);

\end{tikzpicture} 
\caption{•} \label{MatchingExFig}
\end{figure}
\begin{example} 
\label{MatchignEx}
The homogeneous multivariate matching polynomial of the graph $G$ in Figure \ref{MatchingExFig} is given by
$$
\mu(G, \xx \oplus \ww) = x_1x_2x_3x_4 - x_3x_4w_a^2 - x_1x_4w_b^2 - x_2x_4w_c^2 - x_1x_2w_d^2 - x_2x_3w_e^2 + w_a^2w_d^2 + w_b^2w_e^2.
$$
\end{example}
\begin{remark}
Note that $\displaystyle \mu(G, t \mathbf{1} \oplus \mathbf{1}) = \mu(G, t)$ and that the perfect matching polynomial is given by $\mu(G, \mathbf{0} \oplus \mathbf{w})$.
\end{remark} \noindent
In this section we prove Conjecture \ref{glc} in the affirmative for the polynomials $\mu(G, \mathbf{x} \oplus \mathbf{w} )$. We first assert that $\mu(G, \mathbf{x} \oplus \mathbf{w} )$ is indeed a hyperbolic polynomial.  
\begin{lemma}
The polynomial $\mu(G, \mathbf{x} \oplus \mathbf{w} )$ is hyperbolic with respect to $\mathbf{e} = \mathbf{1} \oplus \mathbf{0}$.
\end{lemma}
\begin{proof}
Clearly $\mu(G, \mathbf{1} \oplus \mathbf{0} ) = 1 \neq 0$.
Let $\mathbf{x} \oplus \mathbf{w} \in \RR^G$ and $\lambda_e = w_{e}^2$ for all $e \in  E(G)$. Then
\begin{align*}
\displaystyle \mu(G, t \mathbf{e} - \mathbf{x} \oplus \mathbf{w} ) = \left ( \prod_{v \in V} (t - x_v) \right ) \mu_{\boldsymbol{\lambda}}(G, (t\mathbf{1} - \mathbf{x})^{-1}).
\end{align*} \noindent
Since $\mu_{\boldsymbol{\lambda}}(G, \mathbf{x})$ is real stable by Heilmann-Lieb theorem it follows that the right hand side is real-rooted. Hence $\mu(G, \mathbf{x} \oplus \mathbf{w} )$ is hyperbolic with respect to $\mathbf{e} = \mathbf{1} \oplus \mathbf{0}$.

\end{proof} \noindent
Analogues of the standard recursions for the univariate matching polynomial (see \cite[Thm 1.1]{God}) also hold for $\mu(G, \mathbf{x} \oplus \mathbf{w} )$. In particular the following recursion is used frequently so we give details.
\begin{lemma}
\label{matchrec}
Let $u \in V(G)$. Then the homogeneous multivariate matching polynomial satisfies the recursion 
\begin{align*}
\displaystyle \mu(G, \mathbf{x} \oplus \mathbf{w} ) = x_u \mu(G \setminus u, \mathbf{x} \oplus \mathbf{w} ) - \sum_{v \in N(u)} w_{uv}^2 \mu((G \setminus u) \setminus v, \mathbf{x} \oplus \mathbf{w} ).
\end{align*} \noindent
\end{lemma}
\begin{proof} 
The identity follows by partitioning the matchings $M \in \mathcal{M}(G)$ into two parts depending on whether $u \in V(M)$ or $u \not \in V(M)$. 
Let $f_G(M) = \prod_{v \not \in V(M)} x_v \prod_{e \in M} w_e^2$. Then
\begin{align*}
\displaystyle \mu(G, \mathbf{x} \oplus \mathbf{w}) &= \sum_{M \in \mathcal{M}(G)} (-1)^{|M|} f_G(M) \\ &= \sum_{\substack{M \in \mathcal{M}(G) \\ u \not \in V(M) }} (-1)^{|M|} f_{G}(M) + \sum_{\substack{ M \in \mathcal{M}(G) \\ u \in V(M)} } (-1)^{|M|} f_G(M) \\ &= x_u \sum_{M \in M(G \setminus u)} (-1)^{|M|} f_{G \setminus u}(M) + \sum_{v \in N(u)} \sum_{\substack{M \in \mathcal{M}(G) \\ uv \in M }} (-1)^{|M|} f_G(M) \\ &= x_u \mu(G \setminus u, \mathbf{x} \oplus \mathbf{w} ) - \sum_{v \in N(u)} w_{uv}^2 \sum_{M \in M( (G \setminus u) \setminus v)} (-1)^{|M|} f_{(G \setminus u)\setminus v}(M) \\ &= x_u \mu(G \setminus u, \mathbf{x} \oplus \mathbf{w} ) - \sum_{v \in N(u)} w_{uv}^2 \thinspace \mu((G \setminus u) \setminus v, \mathbf{x} \oplus \mathbf{w}).
\end{align*}
\end{proof} \noindent
Let $G$ be a graph and $u \in V(G)$. The \textit{path tree} $T(G,u)$ is the tree with vertices labelled by paths in $G$ starting at $u$ and where two vertices are joined by an edge if one vertex is labelled by a maximal subpath of the other. 

\begin{example} \noindent \newline
\begin{tabular}{ cc }
$G$ & $T(G,1)$ \\ \\
\begin{tikzpicture}
% create the node

\node[label=above:{ $1$}, draw=black,fill,circle,inner sep=0pt,minimum size=3pt] (x1) at (-2.5,0) {};

\node[label=above:{ $3$}, draw=black,fill,circle,inner sep=0pt,minimum size=3pt] (x3) at (0,0) {};

\node[label=above:{ $5$}, draw=black,fill,circle,inner sep=0pt,minimum size=3pt] (x5) at (2.5,0) {};

\node[label=below:{ $2$}, draw=black,fill,circle,inner sep=0pt,minimum size=3pt] (x2) at (-2.5,-2.5) {};

\node[label=below:{ $4$}, draw=black,fill,circle,inner sep=0pt,minimum size=3pt] (x4) at (0,-2.5) {};

\node[label=below:{ $6$}, draw=black,fill,circle,inner sep=0pt,minimum size=3pt] (x6) at (2.5,-2.5) {};

\draw (x1) -- (x3);
\draw (x3) -- (x5);
\draw (x5) -- (x6);
\draw (x6) -- (x4);
\draw (x4) -- (x2);
\draw (x2) -- (x1);
\draw (x3) -- (x4);

\end{tikzpicture} 
&
\includegraphics[width=65mm]{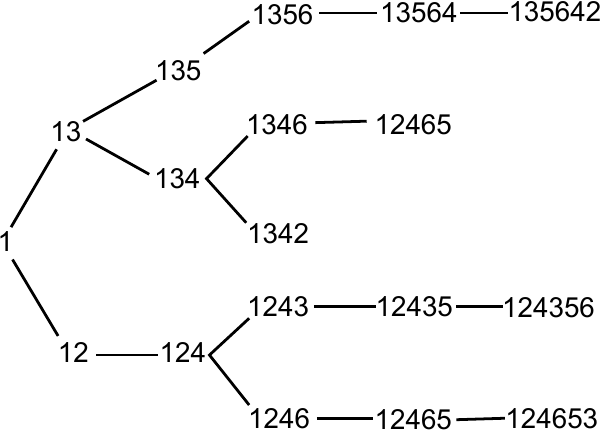}

\end{tabular}
    \label{MatchignExFig}
\end{example}

\begin{definition}
Let $G$ be a graph and $u \in V(G)$. Let $\phi: \mathbb{R}^{T(G,u)} \to \mathbb{R}^{G}$ denote the linear change of variables defined by
\begin{align*}
\displaystyle x_{p} &
\mapsto x_{i_k}, \\ w_{pp'} &\mapsto w_{i_k i_{k+1}},
\end{align*}  \noindent
where $p = i_1 \cdots i_k$ and $p' = i_1 \cdots i_k i_{k+1}$ are adjacent vertices in $T(G,u)$. For every subforest $T \subseteq T(G,u)$, define the polynomial
$$
\eta(T, \xx \oplus \ww) = \mu(T, \phi(\xx' \oplus \ww'))
$$
where $\xx' = (x_p)_{p \in V(T)}$ and $\ww' = (w_e)_{e \in E(T)}$.
\begin{remark}
Note that $\eta(T, \xx \oplus \ww)$ is a polynomial in variables $\xx = (x_v)_{v \in V(G)}$ and $\ww = (w_e)_{e \in E(G)}$.
\end{remark} \noindent
For the univariate matching polynomial we have the following rather unexpected divisibility relation due to Godsil \cite{God2}, 
\begin{align*}
\frac{\mu(G \setminus u, t)}{\mu(G, t)} = \frac{\mu(T(G,u) \setminus u, t)}{\mu(T(G,u), t)}.
\end{align*} \noindent
Below we prove a multivariate analogue of this fact. A similar multivariate analogue was also noted independently by Leake and Ryder \cite{LR}. In fact they were able to find a further generalization to independence polynomials of simplicial graphs. We will revisit their results in Section \ref{sec::indep}. The arguments all closely resemble Godsil's proof for the univariate matching polynomial. For the convenience of the reader we provide the details in our setting.  
\end{definition}
\begin{lemma}
\label{matchquot}
Let $u \in V(G)$. Then
\begin{align*}
\displaystyle \frac{\mu(G \setminus u, \mathbf{x} \oplus \mathbf{w} )}{ \mu(G, \mathbf{x} \oplus \mathbf{w})} = \frac{\eta(T(G,u) \setminus u, \mathbf{x} \oplus \mathbf{w})}{ \eta(T(G,u), \mathbf{x} \oplus \mathbf{w})}.
\end{align*}
\end{lemma}
\begin{proof}
If $G$ is a tree, then $\mu(G, \xx \oplus \ww) = \eta(T(G,u), \xx \oplus \ww)$ and $\mu(G \setminus u, \xx \oplus \ww) = \eta(T(G,u) \setminus u, \xx \oplus \ww)$ so the lemma holds. In particular the lemma holds for all graphs with at most two vertices. We now argue by induction on the number of vertices of $G$. We first claim that
$$
\frac{\eta(T(G,u)\setminus \{ u, uv \}, \xx \oplus \ww )}{\eta(T(G,u) \setminus u, \xx \oplus \ww )} = \frac{\eta(T(G \setminus u, v)\setminus v, \xx \oplus \ww ) }{\eta(T(G \setminus u, v), \xx \oplus \ww)}. 
$$
Let $v \in N(u)$. By examining the path tree $T(G,u)$ we note the following isomorphisms
\begin{align*}
T(G, u) \setminus u &\cong \bigsqcup_{n \in N(u)} T(G \setminus u, n), \\ T(G, u) \setminus \{ u,uv  \} &\cong \left ( \bigsqcup_{\substack{n \in N(u) \\ n \neq v}} T(G \setminus u, n) \right ) \sqcup T(G \setminus u,v) \setminus v, 
\end{align*} \noindent
following from the fact that $T(G\setminus u, n)$ is isomorphic to the connected component of $T(G,u)\setminus u$ which contains the path $un$ in $G$.
By the definition of $\phi$ and the general multiplicative identity 
\begin{align*}
\displaystyle \mu(G \sqcup H, \mathbf{x} \oplus \mathbf{w}) = \mu(G, \mathbf{x} \oplus \mathbf{w})\mu(H, \mathbf{x} \oplus \mathbf{w}),
\end{align*} \noindent 
the above isomorphisms translate to the following identities
\begin{align*}
\displaystyle \eta(T(G, u) \setminus u, \xx \oplus \ww) &= \prod_{n \in N(u)} \eta(T(G \setminus u, n), \xx \oplus \ww), \\  \eta(T(G,u) \setminus \{u,uv\}, \xx \oplus \ww) &= \eta(T(G\setminus u,v) \setminus v, \xx \oplus \ww) \prod_{\substack{n \in N(u) \\ n \neq v}} \eta(T(G \setminus u, n), \xx \oplus \ww),
\end{align*} \noindent
from which the claim follows.
By Lemma \ref{matchrec}, induction, above claim and the definition of $\phi$ we finally get
\begin{align*}
\displaystyle \frac{\mu(G, \xx \oplus \ww)}{\mu(G \setminus u, \xx \oplus \ww)} &= \frac{x_u \mu(G\setminus u, \xx \oplus \ww) - \sum_{v \in N(u)}w_{uv}^2 \mu(G \setminus \{u,v \}, \xx \oplus \ww)  }{\mu(G \setminus u, \xx \oplus \ww)} \\ &= x_u - \sum_{v \in N(u)}w_{uv}^2 \frac{\mu((G \setminus u) \setminus  v, \xx \oplus \ww)}{\mu(G \setminus u, \xx \oplus \ww)} \\ &= x_u - \sum_{v \in N(u)} w_{uv}^2 \frac{\eta(T(G \setminus u, v) \setminus v, \xx \oplus \ww)}{\eta(T(G \setminus u, v), \xx \oplus \ww)} \\ &= x_u - \sum_{v \in N(u)} w_{uv}^2 \frac{\eta(T(G,u)\setminus \{ u, uv \}, \xx \oplus \ww )}{\eta(T(G,u) \setminus u, \xx \oplus \ww )} \\ &= \frac{\eta(T(G,u), \xx \oplus \ww)}{\eta(T(G,u) \setminus u, \xx \oplus \ww )}
\end{align*} \noindent
which is the reciprocal of the desired identity.

\end{proof} \noindent
\begin{lemma} \label{matchdiv}
Let $u \in V(G)$. Then $\mu(G, \xx \oplus \ww)$ divides $\eta(T(G,u), \xx \oplus \ww)$.
\end{lemma}
\begin{proof}
The argument is by induction on the number of vertices of $G$.
Deleting the root $u$ of $T(G,u)$ we get a forest with $|N(u)|$ disjoint components isomorphic to $T(G \setminus u, v)$ respectively for $v \in N(u)$. This gives
\begin{align} \label{proddecomp}
\displaystyle \eta(T(G,u) \setminus u, \mathbf{x} \oplus \mathbf{w}) = \prod_{v \in N(u)} \eta(T(G \setminus u,v), \mathbf{x} \oplus \mathbf{w}).
\end{align} \noindent
Therefore $\eta(T(G \setminus u,v), \mathbf{x} \oplus \mathbf{w})$ divides $\eta(T(G,u) \setminus u, \mathbf{x} \oplus \mathbf{w})$ for all $v \in N(u)$. By induction $\mu(G \setminus u, \mathbf{x} \oplus \mathbf{w})$ divides $\eta(T(G \setminus u,v), \mathbf{x} \oplus \mathbf{w})$ for all $v \in N(u)$. Hence $\mu(G \setminus u, \mathbf{x} \oplus \mathbf{w})$ divides $\eta(T(G,u) \setminus u, \mathbf{x} \oplus \mathbf{w})$, so by Lemma \ref{matchquot}, $\mu(G, \mathbf{x} \oplus \mathbf{w})$ divides $\eta(T(G,u), \mathbf{x} \oplus \mathbf{w})$.

\end{proof}

In \cite{GG} Godsil and Gutman proved the following  relationship between the univariate matching polynomial $\mu(G,t)$ of a graph $G$ and the characteristic polynomial $\chi(A,t)$ of its adjacency matrix $A$
\begin{align*}
\displaystyle \chi(A,t) = \sum_{C} (-2)^{\text{comp}(C)} \mu(G \setminus C, t),
\end{align*} \noindent
where the sum ranges over all subgraphs $C$ (including $C = \emptyset$) in which each component is a cycle of degree $2$ and $\text{comp}(C)$ is the number of connected components of $C$. In particular if $T$ is a tree, then the only such subgraph is $C = \emptyset$ and therefore
\begin{align*}
\displaystyle \chi(A,t) = \mu(T,t).
\end{align*} \noindent
Next we will derive a multivariate analogue of this relationship for trees.
\begin{lemma}
\label{treedetrep}
Let $T = (V,E)$ be a tree. Then $\mu(T, \mathbf{x} \oplus \mathbf{w})$ has a definite determinantal representation.
\end{lemma}
\begin{proof}
Let $X = \text{diag}(\mathbf{x})$ and $A = (A_{ij})$ be the matrix 
$$
A_{ij} = \begin{cases} w_{ij} & \text{ if } ij \in E(T) \\ 0 & \text{ otherwise} \end{cases}
$$ 
for all $i,j \in V(T)$. If $\sigma \in \text{Sym}(V(T))$ is an involution (i.e $\sigma^2 = id$), then clearly $A_{j \sigma(j)} = w_{j\sigma(j)} = A_{\sigma(j) \sigma^2(j)}$ since $A$ is symmetric.
Hence by acyclicity of trees we have that
\begin{align*}
\displaystyle  \det(X + A) &= \sum_{\sigma \in \text{Sym}(V(T))} \text{sgn}(\sigma) \prod_{i \in V(T)} (X_{i \sigma(i)} + A_{i \sigma(i)}) \\ &=  \sum_{S \subseteq V(T)} \prod_{i \in V(T) \setminus S} x_i \sum_{\substack{\sigma \in \text{Sym}(S) \\ \sigma(j) \neq j \thinspace \forall j \in S \\ \sigma^2 = \text{id}}} \text{sgn}(\sigma) \prod_{j \in S} A_{j \sigma(j)} \\ &= \sum_{S \subseteq V(T)} \prod_{i \in V(T) \setminus S} x_i \sum_{\substack{M \in \mathcal{M}(T[S])\\ M \text{ perfect}}}(-1)^{|M|} \prod_{jk \in M} w_{jk}^2  \\ &= \sum_{M \in \mathcal{M}(T)} (-1)^{|M|} \prod_{i \not \in M} x_i \prod_{jk \in M} w_{jk}^2 && \\ &= \mu(T, \mathbf{x} \oplus \mathbf{w}).
\end{align*} \noindent
Write
\begin{align*}
\displaystyle X + A =  \sum_{i \in V(T)} x_i E_{ii} + \sum_{ij \in E(T)} w_{ij} (E_{ij} + E_{ji}),
\end{align*} \noindent
where $\{E_{ij} : i,j \in V(T)\}$ denotes the standard basis for the vector space of all real $|V(T)| \times |V(T)|$ matrices. Evaluated at $\ee = \mathbf{1} \oplus \mathbf{0}$ we obtain the identity matrix $I$ which is positive definite.

\end{proof}
\begin{remark}
The proof of Lemma \ref{treedetrep} is not dependent on $T$ being connected so the statement remains valid for arbitrary undirected acyclic graphs (i.e. forests). 
\end{remark} \noindent
We now have all the ingredients to prove our main theorem.
\begin{theorem}
\label{matchconespec}
The hyperbolicity cone of $\mu(G, \mathbf{x} \oplus \mathbf{w})$ is spectrahedral. 
\end{theorem}
\begin{proof}
%\todo{Proof is a bit messy. Look to clean it up.}
The proof is by induction on the number of vertices of $G$. For the base case we have $\mu(G, \xx \oplus \ww) = x_v$, so $\Lambda_{+} = \{ x \in \RR : x \geq 0 \}$ which is clearly spectrahedral. Assume $G$ contains more than one vertex. If $G = G_1 \sqcup G_2$ for some non-empty graphs $G_1,G_2$, then $\Lambda_{++}(\mu(G_i, \xx, \oplus \ww))$ is spectrahedral by induction for $i = 1,2$. Therefore
\begin{align*}
\displaystyle \Lambda_{++}(\mu(G, \xx, \oplus \ww)) &= \Lambda_{++}(\mu(G_1 \sqcup G_2, \xx, \oplus \ww)) \\&= \Lambda_{++} \left (\mu(G_1, \xx, \oplus \ww)\mu(G_1, \xx, \oplus \ww) \right )  \\ &= \Lambda_{++}(\mu(G_1, \xx, \oplus \ww))  \cap \Lambda_{++}(\mu(G_2, \xx, \oplus \ww))
\end{align*} \noindent
showing that $\Lambda_{++}(\mu(G, \xx, \oplus \ww))$ is spectrahedral.
We may therefore assume $G$ is connected.
Let $u \in V(G)$. Since $G$ is connected and has size greater than one, $N(u) \neq \emptyset$. By Lemma \ref{matchdiv} we may define the polynomial
$$
q_{G,u}(\xx \oplus \ww) = \frac{\eta(T(G,u), \xx \oplus \ww)}{\mu(G, \xx \oplus \ww)}
$$
for each graph $G$ and $u \in V(G)$.
We want to show that 
$$
\Lambda_{++}(\mu(G, \xx \oplus \ww)) \subseteq \Lambda_{++}(q_{G,u}(\xx \oplus \ww)).
$$
By Lemma \ref{matchquot} we have that
\begin{align*}
\displaystyle q_{G,u}(\mathbf{x} \oplus \mathbf{w}) \mu(G \setminus u, \mathbf{x} \oplus \mathbf{w} ) = \eta(T(G,u) \setminus u, \mathbf{x} \oplus \mathbf{w}).
\end{align*} \noindent
Fixing $v \in N(u)$ it follows using (\ref{proddecomp}) that
\begin{align*}
\displaystyle \frac{q_{G,u}(\mathbf{x} \oplus \mathbf{w}) }{q_{G \setminus u,v}(\mathbf{x} \oplus \mathbf{w})} &=  \frac{q_{G,u}(\mathbf{x} \oplus \mathbf{w}) \mu(G \setminus u, \mathbf{x} \oplus \mathbf{w}) }{q_{G \setminus u,v}(\mathbf{x} \oplus \mathbf{w}) \mu(G \setminus u, \mathbf{x} \oplus \mathbf{w})} \\ &=  \frac{\eta(T(G,u) \setminus u, \mathbf{x} \oplus \mathbf{w} )}{ \eta(T(G \setminus u, v), \mathbf{x} \oplus \mathbf{w})} \\ &= \prod_{w \in N(u) \setminus v} \eta(T(G \setminus u, w), \mathbf{x} \oplus \mathbf{w}) \\ &= \prod_{w \in N(u) \setminus v} q_{G \setminus u,w}(\mathbf{x} \oplus \mathbf{w}) \mu(G \setminus u, \mathbf{x} \oplus \mathbf{w}).
\end{align*} \noindent
Note that 
$$\frac{\partial}{\partial x_u} \mu(G, \mathbf{x} \oplus \mathbf{w}) = \mu(G \setminus u, \mathbf{x} \oplus \mathbf{w}).
$$
Therefore by Lemma \ref{derivhyp},
\begin{align*}
\displaystyle \Lambda_{++}(\mu(G, \mathbf{x} \oplus \mathbf{w})) \subseteq \Lambda_{++}(\mu(G \setminus u, \mathbf{x} \oplus \mathbf{w})) \subseteq \Lambda_{++}(q_{G \setminus u,w}(\mathbf{x} \oplus \mathbf{w}))
\end{align*} \noindent
for all $w \in N(u)$ where the last inclusion follows by inductive hypothesis. Hence
\begin{align*}
\displaystyle \Lambda_{++}(\mu(G, \mathbf{x} \oplus \mathbf{w})) &\subseteq \bigcap_{w \in N(u)} \Lambda_{++} (q_{G \setminus u,w}(\xx \oplus \ww) \cap \Lambda_{++} \left (\mu(G \setminus u, \xx \oplus \ww \right ) \\ &=  \Lambda_{++} \left ( q_{v, G \setminus u}(\mathbf{x} \oplus \mathbf{w}) \prod_{w \in N(u) \setminus v} q_{w, G \setminus u}(\mathbf{x} \oplus \mathbf{w}) \mu(G \setminus u, \mathbf{x} \oplus \mathbf{w})  \right ) \\ &= \Lambda_{++}(q_{G,u}(\xx \oplus \ww)).
\end{align*} \noindent
Finally by Lemma \ref{treedetrep}, $\eta(T(G,u), \mathbf{x} \oplus \mathbf{w})$ has a definite determinantal representation. Hence the theorem follows by induction.
 
\end{proof} \noindent
\begin{remark}
To show that a hyperbolic polynomial $h$ has a spectrahedral hyperbolicity cone it is by Theorem \ref{matchconespec} sufficient to show that $h$ can be realized as a factor of a matching polynomial $\mu(G, \mathbf{x} \oplus \mathbf{w})$ with $\Lambda_{++}(h, \mathbf{e}) \subseteq \Lambda_{++} \left ( \frac{\mu(G, \mathbf{x} \oplus \mathbf{w})}{h}, \mathbf{e} \right )$ (possibly after a linear change of variables).
\end{remark} \noindent

\begin{figure} 
    \centering
\begin{tikzpicture}
% create the node
\node[draw=white,minimum size=5cm,regular polygon,regular polygon sides=6] (p) {};

\node[label=below right:{\scriptsize $x_{n+1}$}, draw=black,fill,circle,inner sep=0pt,minimum size=3pt] (x1) {};

\node[label={\scriptsize $x_2$}, draw=black,fill,circle,inner sep=0pt,minimum size=3pt] (x2) at (p.corner 1) {};

\node[label={\scriptsize $x_{n}$}, draw=black,fill,circle,inner sep=0pt,minimum size=3pt] (xn) at (p.corner 2) {};

\node[label=left:{\scriptsize $x_6$}, draw=black,fill,circle,inner sep=0pt,minimum size=3pt] (x6) at (p.corner 3) {};

\node[label=below:{\scriptsize $x_5$}, draw=black,fill,circle,inner sep=0pt,minimum size=3pt] (x5) at (p.corner 4) {};

\node[label=below:{\scriptsize $x_4$}, draw=black,fill,circle,inner sep=0pt,minimum size=3pt] (x4) at (p.corner 5) {};

\node[label=right:{\scriptsize $x_{3}$}, draw=black,fill,circle,inner sep=0pt,minimum size=3pt] (x3) at (p.corner 6) {};

\coordinate (anchor1) at (-0.4,0.7);
\coordinate (anchor2) at (-0.8,0);

\node (w1) at (0.9, 1) {\scriptsize $w_1$};
\node (w2) at (1.5, -0.2) {\scriptsize $w_2$};
\node (w3) at (0.5, -1.4) {\scriptsize $w_3$};
\node (w4) at (-1, -1.2) {\scriptsize $w_4$};
\node (w5) at (-1.5, 0.15) {\scriptsize $w_5$};
\node (wn) at (-0.4, 1.2) {\scriptsize $w_n$};

\draw (x1) -- (x2);
\draw (x1) -- (x3);
\draw (x1) -- (x4);
\draw (x1) -- (x5);
\draw (x1) -- (x6);
\draw (x1) -- (xn);
\draw[dotted]
      (anchor1) to[out=140,in=140] (anchor2);

\end{tikzpicture} 
\caption{The star graph $S_n$ labelled by vertex and edge variables}
    \label{stargraph}
\end{figure}
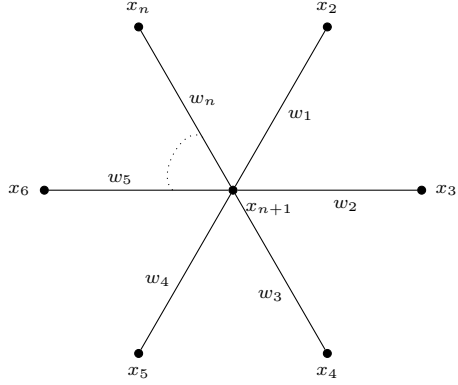

The \textit{elementary symmetric polynomial} $e_d(\xx) \in \RR[x_1, \dots, x_n]$ of degree $d$ in $n$ variables is defined by
$$
e_d(\xx) = \sum_{\substack{S \subseteq [n] \\ |S| = d}} \prod_{i \in S} x_i.
$$
The polynomials $e_d(\xx)$ are hyperbolic (in fact stable) as a consequence of e.g Grace-Walsh-Szeg\H{o} theorem (see \cite[Thm 15.4]{Mar}).
\begin{example}
The star graph, denoted $S_n$, is given by the complete bipartite graph $K_{1,n}$ with $n+1$ vertices. As an application of Theorem \ref{matchconespec} we show that several well-known instances of hyperbolic polynomials have spectrahedral hyperbolicity cones by realizing them as factors of the multivariate matching polynomial of $S_n$ under some linear change of variables.
With notation as in Figure \ref{stargraph}, using the recursion in Lemma \ref{matchrec}, the multivariate matching polynomial of $S_n$ is given by 
\begin{align*}
\displaystyle \mu(S_n, \mathbf{x} \oplus \mathbf{w}) = \prod_{i = 1}^{n+1} x_i - \sum_{i=1}^n w_i^2 \prod_{\substack{j = 1 \\ j \neq i}}^{n} x_j. 
\end{align*}
\begin{enumerate}
\item For $h(\mathbf{x}) = e_{n-1}(\mathbf{x})$ consider the linear change of variables $x_n \mapsto -x_n$ and $w_i \mapsto x_n$ for $i = 1, \dots, n-1$. Then $\mu(S_{n-1}, \mathbf{x} \oplus \mathbf{w}) \mapsto -x_n e_{n-1}(\mathbf{x})$. Clearly $\Lambda_{++}(e_{n-1}(\mathbf{x}), \mathbf{1}) \subseteq \Lambda_{++}(x_n, \mathbf{1})$. The spectrahedrality of $\Lambda_{++}(e_{n-1}(\xx), \mathbf{1})$ was first proved by Sanyal in \cite{S}.
\item For $h(\mathbf{x}) = e_{2}(\mathbf{x})$ consider the linear change of variables $x_i \mapsto e_1(\mathbf{x})$ and $w_i \mapsto x_i$ for $i = 1, \dots, n+1$. Then $\mu(S_{n}, \mathbf{x} \oplus \mathbf{w}) \mapsto 2e_1(\mathbf{x})^{n-1}e_{2}(\mathbf{x})$. Since $D_{\mathbf{1}} e_2(\xx) = (n-1)e_1(\xx)$, Lemma \ref{derivhyp} implies that $\Lambda_{++}(e_2(\xx), \mathbf{1}) \subseteq \Lambda_{++}(e_1(\xx), \mathbf{1})$. Hence $\Lambda_{++}(e_2(\xx), \mathbf{1})$ is spectrahedral.
\item Let $h(\mathbf{x}) = x_n^2 - x_{n-1}^2 - \cdots - x_{1}^2$. Recall that $\Lambda_{++}(h, \mathbf{e})$ is the \textit{Lorentz cone} where $\mathbf{e} = (0, \dots, 0,1)$. Consider the linear change of variables $x_i \mapsto x_n$ and $w_i \mapsto x_{i}$ for $i = 1, \dots, n$. Then $\mu(S_{n-1}, \mathbf{x} \oplus \mathbf{w}) \mapsto x_n^n - \sum_{i = 1}^{n-1} x_{i}^2 x_n^{n-2} = x_n^{n-2}h(\mathbf{x})$. Clearly $\Lambda_{++}(h, \mathbf{e}) \subseteq \Lambda_{++}(x_n^{n-2}, \mathbf{e})$. Hence the Lorentz cone is spectrahedral. Of course this (and the preceding example) also follow from the fact that all quadratic hyperbolic polynomials have spectrahedral hyperbolicity cone \cite{NT}.
\end{enumerate}
\end{example} \noindent
Hyperbolicity cones of elementary symmetric polynomials have been studied by Zinchenko \cite{Z}, Sanyal \cite{S} and Br\"and\'en \cite{B2}. Br\"and\'en proved that all hyperbolicity cones of elementary symmetric polynomials are spectrahedral. As an application of Theorem \ref{matchconespec} we give a new proof of this fact using matching polynomials.

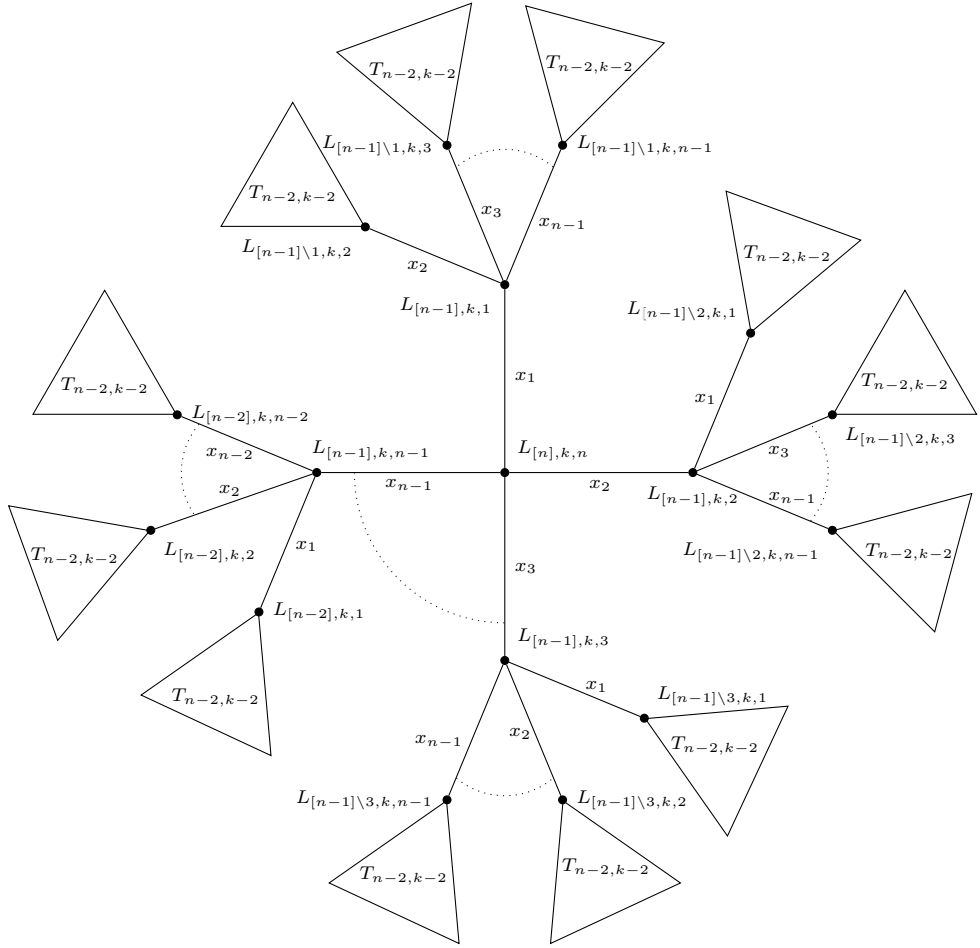
\begin{figure} 
    \centering
\begin{tikzpicture}
% create the node

\node[label=above right:{\scriptsize $L_{[n],k,n}$}, draw=black,fill,circle,inner sep=0pt,minimum size=3pt] (v1) at (0,0) {};

\node[label=below:{\scriptsize $L_{[n-1],k,2}$}, draw=black,fill,circle,inner sep=0pt,minimum size=3pt] (v12) at (2.5,0) {};

\node[label=above right:{\scriptsize $L_{[n-1],k,3}$}, draw=black,fill,circle,inner sep=0pt,minimum size=3pt] (v13) at (0,-2.5) {};

\node[label=below left:{\scriptsize $L_{[n-1],k,1}$}, draw=black,fill,circle,inner sep=0pt,minimum size=3pt] (v11) at (0,2.5) {};

\node[label=above right:{ \hspace{-1em} \scriptsize $L_{[n-1],k,n-1}$}, draw=black,fill,circle,inner sep=0pt,minimum size=3pt] (v1n) at (-2.5,0) {};

\node[draw=white,minimum size=4cm,regular polygon,regular polygon sides=8] (p1) at (v12) {};

\node[label=above left:{\scriptsize $L_{[n-1]\setminus 2,k,1}$}, draw=black,fill,circle,inner sep=0pt,minimum size=3pt] (v21) at (p1.corner 1) {};

\node[draw=black,minimum size=2.2cm,regular polygon,regular polygon sides=3,rotate=40, anchor=corner 2] (t1) at (v21) {};

\node[] (T21) at ([xshift=3pt]t1) {\scriptsize $T_{n-2,k-2}$};

\node[label=below right:{\scriptsize $L_{[n-1]\setminus 2,k,3}$}, draw=black,fill,circle,inner sep=0pt,minimum size=3pt] (v22) at (p1.corner 8) {};

\node[draw=black,minimum size=2.2cm,regular polygon,regular polygon sides=3,rotate=0, anchor=corner 2] (t2) at (v22) {};

\node[] (T22) at ([yshift=-3.5pt]t2) {\scriptsize $T_{n-2,k-2}$};

\node[label=below left:{\scriptsize $L_{[n-1]\setminus 2,k,n-1}$}, draw=black,fill,circle,inner sep=0pt,minimum size=3pt] (v23) at (p1.corner 7) {};

\node[draw=black,minimum size=2.2cm,regular polygon,regular polygon sides=3,rotate=-45, anchor=corner 2] (t3) at (v23) {};

\node[] (T23) at ([xshift=-1.5pt]t3) {\scriptsize $T_{n-2,k-2}$};

\node[] (e11) at ([xshift= 8pt, yshift=-35]v11) {\scriptsize $x_1$};

\node[] (e12) at ([xshift=-35pt, yshift=-5]v12) {\scriptsize $x_2$};

\node[] (e13) at ([xshift=8pt, yshift=35]v13) {\scriptsize $x_{3}$};

\node[] (e1n) at ([xshift=35pt, yshift=-5]v1n) {\scriptsize $x_{n-1}$};

\node[] (e21) at ([xshift=-17pt, yshift=-25]v21) {\scriptsize $x_1$};

\node[] (e22) at ([xshift=-20pt, yshift=-14]v22) {\scriptsize $x_3$};

\node[] (e23) at ([xshift=-15pt, yshift=12]v23) {\scriptsize $x_{n-1}$};

\node[draw=white,minimum size=4cm,regular polygon,regular polygon sides=8] (p2) at (v1n) {};

\node[label=right:{\scriptsize $L_{[n-2],k,1}$}, draw=black,fill,circle,inner sep=0pt,minimum size=3pt] (vn1) at (p2.corner 5) {};

\node[draw=black,minimum size=2.2cm,regular polygon,regular polygon sides=3,rotate=-25, anchor=corner 1] (t4) at (vn1) {};

\node[label=below right:{\scriptsize $L_{[n-2],k,2}$}, draw=black,fill,circle,inner sep=0pt,minimum size=3pt] (vn2) at ([xshift=-10pt] p2.corner 4) {};

\node[draw=black,minimum size=2.2cm,regular polygon,regular polygon sides=3,rotate=-70, anchor=corner 1] (t5) at (vn2) {};

\node[label=right:{\scriptsize $L_{[n-2],k,n-2}$}, draw=black,fill,circle,inner sep=0pt,minimum size=3pt] (vn3) at (p2.corner 3) {};

\node[draw=black,minimum size=2.2cm,regular polygon,regular polygon sides=3,rotate=-120, anchor=corner 1] (t6) at (vn3) {};

\node[] (en1) at ([xshift= 18pt, yshift=25]vn1) {\scriptsize $x_1$};

\node[] (en2) at ([xshift= 30pt, yshift=15]vn2) {\scriptsize $x_2$};

\node[] (en3) at ([xshift= 20pt, yshift=-15]vn3) {\scriptsize $x_{n-2}$};

\node[] (Tn1) at ([yshift=-3.5pt, xshift=-3pt]t4) {\scriptsize $T_{n-2,k-2}$};

\node[] (Tn2) at ([yshift=1pt, xshift=1pt]t5) {\scriptsize $T_{n-2,k-2}$};

\node[] (Tn3) at ([yshift=-5pt, xshift=0pt]t6) {\scriptsize $T_{n-2,k-2}$};

\node[draw=white,minimum size=4cm,regular polygon,regular polygon sides=8] (p3) at (v11) {};

\node[label=below left:{\scriptsize $L_{[n-1]\setminus 1,k,2}$}, draw=black,fill,circle,inner sep=0pt,minimum size=3pt] (v111) at (p3.corner 3) {};

\node[] (e111) at ([xshift= 20pt, yshift=-15]v111) {\scriptsize $x_2$};

\node[draw=black,minimum size=2.2cm,regular polygon,regular polygon sides=3,rotate=0, anchor=corner 3] (t7) at (v111) {};

\node[] (T11) at ([yshift=-3.5pt, xshift=0pt]t7) {\scriptsize $T_{n-2,k-2}$};

\node[label=left:{\scriptsize $L_{[n-1]\setminus 1,k,3}$}, draw=black,fill,circle,inner sep=0pt,minimum size=3pt] (v112) at (p3.corner 2) {};

\node[] (e112) at ([xshift= 17pt, yshift=-25]v112) {\scriptsize $x_3$};

\node[draw=black,minimum size=2.2cm,regular polygon,regular polygon sides=3,rotate=-40, anchor=corner 3] (t8) at (v112) {};

\node[] (T12) at ([yshift=-2pt, xshift=-2pt]t8) {\scriptsize $T_{n-2,k-2}$};

\node[label=right:{\scriptsize $L_{[n-1]\setminus 1,k,n-1}$}, draw=black,fill,circle,inner sep=0pt,minimum size=3pt] (v113) at (p3.corner 1) {};

\node[] (e113) at ([xshift= 0pt, yshift=-30]v113) {\scriptsize $x_{n-1}$};

\node[draw=black,minimum size=2.2cm,regular polygon,regular polygon sides=3,rotate=45, anchor=corner 2] (t9) at (v113) {};

\node[] (T13) at ([yshift=0pt, xshift=2pt]t9) {\scriptsize $T_{n-2,k-2}$};

\node[draw=white,minimum size=4cm,regular polygon,regular polygon sides=8] (p4) at (v13) {};

\node[label=above right:{\scriptsize $L_{[n-1]\setminus 3,k,1}$}, draw=black,fill,circle,inner sep=0pt,minimum size=3pt] (v131) at (p4.corner 7) {};

\node[] (e131) at ([xshift= -18pt, yshift=12]v131) {\scriptsize $x_{1}$};

\node[draw=black,minimum size=2.2cm,regular polygon,regular polygon sides=3,rotate=65, anchor=corner 1] (t10) at (v131) {};

\node[] (T31) at ([yshift=3pt, xshift=-2pt]t10) {\scriptsize $T_{n-2,k-2}$};

\node[label=right:{\scriptsize $L_{[n-1]\setminus 3,k,2}$}, draw=black,fill,circle,inner sep=0pt,minimum size=3pt] (v132) at (p4.corner 6) {};

\node[] (e132) at ([xshift= -16pt, yshift=25]v132) {\scriptsize $x_{2}$};

\node[draw=black,minimum size=2.2cm,regular polygon,regular polygon sides=3,rotate=25, anchor=corner 1] (t11) at (v132) {};

\node[] (T32) at ([yshift=3pt, xshift=3pt]t11) {\scriptsize $T_{n-2,k-2}$};

\node[label=left:{\scriptsize $L_{[n-1]\setminus 3,k,n-1}$}, draw=black,fill,circle,inner sep=0pt,minimum size=3pt] (v133) at (p4.corner 5) {};

\node[] (e133) at ([xshift= -3pt, yshift=22]v133) {\scriptsize $x_{n-1}$};

\node[draw=black,minimum size=2.2cm,regular polygon,regular polygon sides=3,rotate=-25, anchor=corner 1] (t12) at (v133) {};

\node[] (T33) at ([yshift=0pt, xshift=-3pt]t12) {\scriptsize $T_{n-2,k-2}$};

\draw (v1) -- (v12);
\draw (v1) -- (v13);
\draw (v1) -- (v11);
\draw (v1) -- (v1n);

\draw (v12) -- (v21);
\draw (v12) -- (v22);
\draw (v12) -- (v23);

\draw (v1n) -- (vn1);
\draw (v1n) -- (vn2);
\draw (v1n) -- (vn3);

\draw (v11) -- (v111);
\draw (v11) -- (v112);
\draw (v11) -- (v113);

\draw (v13) -- (v131);
\draw (v13) -- (v132);
\draw (v13) -- (v133);

\draw [dotted,domain=230:310] plot ({cos(\x)}, {-3.3 + sin(\x)});

\draw [dotted,domain=50:130] plot ({cos(\x)}, {3.3 + sin(\x)});

\draw [dotted,domain=-40:40] plot ({3.3 + cos(\x)}, {sin(\x)});

\draw [dotted,domain=180:270] plot ({2*cos(\x)}, {2*sin(\x)});

\draw [dotted,domain=140:215] plot ({-3.3 + cos(\x)}, {sin(\x)});

\end{tikzpicture} 
\caption{The length k-truncated path tree $T_{n,k}$ of $K_n$ labelled by linear change of variables.}
    \label{truncpathtree}
\end{figure}

\begin{theorem} \label{elemspec}
Hyperbolicity cones of elementary symmetric polynomials are spectrahedral.
\end{theorem}  
\begin{proof}
For a subset $S \subseteq [n]$ we shall use the notation
\begin{align*}
\displaystyle e_k(S) = \sum_{\substack{T \subseteq S \\ |T| = k}} \prod_{j \in T} x_j.
\end{align*} \noindent
We show that $e_k(\mathbf{x}) = e_k([n])$ divides the multivariate matching polynomial of the length $k$-truncated path tree $T_{n,k}$ of the complete graph $K_n$ rooted at a vertex $v$ after a linear change of variables. 
Let $(C_k)_{k \geq 0}$ denote the real sequence defined by
\begin{align*}
\displaystyle C_{0} = 1, \thickspace C_{1} = 1, \thickspace C_k = \prod_{j = 0}^{\lfloor k/2 \rfloor - 1} \frac{k - 2j}{k - 2j-1} \hspace{0.2cm} \text{ for } k \geq 2,
\end{align*} \noindent
so that
\begin{align*}
\displaystyle C_{k}C_{k-1} = k \text{ for all } k \geq 1.
\end{align*} \noindent
Consider the family 
\begin{align*}
\displaystyle M_{S,k,i} = \mu(T_{|S|, k}, \phi_{S,k,i}(\mathbf{x} \oplus \mathbf{w})) 
\end{align*} \noindent
of multivariate matching polynomials where $i \in S, \thinspace k \in \mathbb{N}$ and $\phi_{S,k,i}$ is the linear change of variables defined recursively (see Fig \ref{truncpathtree}) via
\begin{enumerate}
\item $\phi_{S,0,i}$ is the map $x_v \mapsto e_1(S)$ for all $S \subseteq [n]$ and $i \in S$.
\item $x_{v} \mapsto L_{S,k,i}$ if $k \geq 1$ where 
$$
L_{S,k,i} = \frac{1}{C_{k-1}}e_1(S \setminus i) + C_k x_i
$$ 
and $x_v$ is the variable corresponding to the root of $T_{n,k}$.
\item $w_{e_j} \mapsto x_j$ for $j \in S \setminus i$ where $w_{e_j}$ are the variables corresponding to the edges $e_j$ incident to the root of $T_{n,k}$.
\item For each $j \in S \setminus i$ make recursively the linear substitutions $\phi_{S \setminus i, k-1, j}$ respectively to the variables corresponding to the $j$-indexed copies of the subtrees of $T_{n,k}$ isomorphic to $T_{n-1,k-1}$.
\end{enumerate} \noindent
We claim
\begin{align*}
\displaystyle
M_{S,0,i} &= e_1(S), \\   
M_{S,k,i} &= \frac{C_k e_k(S)}{e_{k-1}(S \setminus \{i\})} \prod_{j \in S \setminus \{i\}} M_{S \setminus i, k-1, j}
\end{align*} \noindent
for all $S \subseteq [n], \thinspace i \in S$ and $k \in \mathbb{N}$ by induction on $k$. Clearly $M_{S,0,i} = e_1(S)$ since $\mu(T_{n,0}, \mathbf{x} \oplus \mathbf{w}) = x_v$.
By Lemma \ref{matchrec} and induction we have
\begin{align*}
\displaystyle &M_{S,k,i} \\ &= L_{S,k,i} \prod_{s \in S \setminus \{i \}} M_{S \setminus i, k - 1, s} - \sum_{j \in S \setminus i} x_j^2 \prod_{s \in S \setminus \{i,j\}} M_{S \setminus i, k - 1, s} \thinspace M_{S \setminus \{i,j \}, k - 2, s}
\\ &= \left ( \frac{1}{C_{k-1}}e_1(S \setminus i) + C_kx_i - \sum_{j \in S \setminus \{i\}} x_j^2 \thinspace \frac{e_{k-2}(S \setminus \{i,j \})}{C_{k-1}e_{k-1}(S \setminus i)} \right )  \prod_{s \in S \setminus \{i\}} M_{S \setminus i, k-1, s} \\ &=
 \frac{1}{e_{k-1}(S \setminus i)} \left ( \left (\frac{1}{C_{k-1}}e_1(S \setminus i) + C_k x_i \right ) e_{k-1}(S \setminus i) - \frac{1}{C_{k-1}}\sum_{j \in S \setminus i} x_j^2 e_{k-2}(S \setminus \{i,j \})  \right ) \\ & \times  \prod_{s \in S \setminus \{i\}} M_{S \setminus i, k-1, s} \\ &= \frac{1}{e_{k-1}(S \setminus i)} \left ( \frac{k}{C_{k-1}} e_{k}(S \setminus i) + C_kx_i e_{k-1}(S \setminus i) \right ) \prod_{s \in S \setminus \{i\}} M_{S \setminus i, k-1, s} \\ &=  \frac{C_k e_k(S)}{e_{k-1}(S \setminus i)} \prod_{s \in S \setminus \{i\}} M_{S \setminus i, k-1, s}.
\end{align*} \noindent
Unwinding the above recursion it follows that
$M_{S,k,i}$ is of the form
\begin{align*}
\displaystyle M_{S,k,i} = C e_k(S) \prod_{\substack{T \subseteq S \setminus i \\ |T| > |S| - k}} e_{k + |T| - |S|}(T)^{\alpha_T}
\end{align*} \noindent
for some constant $C$ and exponents $\alpha_T \in \mathbb{N}$ . Taking $S = [n]$ we thus see that $e_k(\mathbf{x})$ is a factor of the multivariate matching polynomial $M_{[n], k, n}$. It remains to show that
\begin{align*}
\displaystyle \Lambda_{++}(e_k(\mathbf{x}), \mathbf{1}) \subseteq \Lambda_{++} \left (\frac{M_{[n],k,n}}{e_k(\mathbf{x})}, \mathbf{1} \right ).
\end{align*} \noindent
for all $k \leq n$.
By Lemma \ref{derivhyp} above inclusion follows from the fact that
\begin{align*}
\displaystyle \Lambda_{++}(e_k(S), \mathbf{1}) \subseteq \Lambda_{++}(e_{k-1}(S), \mathbf{1})
\end{align*} \noindent
for all $k \geq 1$ since $D_{\mathbf{1}} e_{k}(S) = (|S|-k)e_{k-1}(S)$, and from the fact that
\begin{align*}
\displaystyle \Lambda_{++}(e_k(S), \mathbf{1}) \subseteq \Lambda_{++} \left (e_k(T), \mathbf{1} \right )
\end{align*} \noindent
for all $T \subseteq S$ since $e_k(T) = \left ( \prod_{i \in S \setminus T} \frac{\partial}{\partial x_i} \right ) e_k(S)$. Hence $\Lambda_{++}(e_k(\mathbf{x}), \mathbf{1})$ is spectrahedral by Theorem \ref{matchconespec}.

\end{proof} \noindent
\section{Hyperbolicity cones of multivariate independence polynomials} \label{sec::indep} \noindent
A subset $I \subseteq V(G)$ is \textit{independent} if no two vertices of $I$ are adjacent in $G$. Let $\mathcal{I}(G)$ denote the set of all independent sets in $G$ and $i(G,k)$ denote the number of independent sets in $G$ of size $k$. By convention $i(G,0) = 1$. The (univariate) \textit{independence polynomial} is defined by
$$
I(G,t) = \sum_{k \geq 0} i(G,k) t^k.
$$ 
The \textit{line graph} $L(G)$ of $G$ is the graph having vertex set $E(G)$ and where two vertices in $L(G)$ are adjacent if and only if the corresponding edges in $G$ are incident. It follows that $\mu(G,t) = t^{|V(G)|}I(L(G),-t^{-2})$. Therefore the independence polynomial can be viewed as a generalization of the matching polyomial. In contrast to the matching polynomial, the independence polynomial of a graph is not real-rooted in general. However Chudnovsky and Seymour \cite{CS} proved that $I(G,t)$ is real-rooted if $G$ is claw-free, that is, if $G$ has no induced subgraph isomorphic to the complete bipartite graph $K_{1,3}$.  The theorem was later generalized by Engstr\"om to graphs with weighted vertices.
\begin{theorem}[Engstr\"om \cite{Eng}] \label{engstrom}
Let $G$ be a claw-free graph and $\boldsymbol{\lambda} = (\lambda_v)_{v \in  V(G)}$ a sequence of non-negative vertex weights. Then the polynomial
$$
I_{\boldsymbol{\lambda}}(G,t) =  \sum_{I \in \mathcal{I}(G)} \left ( \prod_{v \in I} \lambda_v \right ) t^{|I|}
$$
is real-rooted.
\end{theorem} \noindent
A full characterization of the graphs for which $I(G,t)$ is real-rooted remains an open problem. 

A natural multivariate analogue of the independence polynomial is given by
$$
I(G, \xx) = \sum_{I \in \mathcal{I}(G)}  \prod_{v \in I} x_v.
$$
Leake and Ryder \cite{LR} define a strictly weaker notion of stability which they call same-phase stability. A polynomial $p(\zz) \in \RR[z_1, \dots, z_n]$ is \textit{(real) same-phase stable} if for every $\xx \in \RR_+^n$, the univariate polynomial $p(t \xx)$ is real-rooted. The authors prove that $I(G, \xx)$ is same-phase stable if and only if $G$ is claw-free. In fact the same-phase stability of $I(G, \xx)$ is an immediate consequence of Theorem \ref{engstrom}. 

The added variables in a homogeneous multivariate independence polynomial should preferably have labels carrying combinatorial meaning in the graph. For line graphs it is additionally desirable to maintain a natural correspondence with the homogeneous multivariate matching polynomial $\mu(G, \xx \oplus \ww)$. Unfortunately we have not found a hyperbolic definition that satisfies both of the above properties. We have thus settled for the following definition.
\begin{definition}
Let $\xx = (x_v)_{v \in V}$ and $t$ be indeterminates. 
Define the \textit{homogeneous multivariate independence polynomial} $I(G,\xx \oplus t) \in \RR[\xx,t]$ by
$$
I(G,\xx \oplus t) = \sum_{I \in \mathcal{I}(G)} (-1)^{|I|} \left ( \prod_{v \in I} x_v^2 \right ) t^{2|V(G)| -2 |I|}.
$$
\end{definition} \noindent
\begin{lemma} \label{indephyp}
If $G$ is a claw-free graph, then $I(G, \xx,t)$ is a hyperbolic polynomial with respect to $\ee = (0, \dots, 0,1) \in \RR^{V(G)} \times \RR$.
\end{lemma}
\begin{proof}
First note that $I(G, \ee) = 1 \neq 0$. Let $\xx \oplus t \in \RR^{V(G)} \times \RR$ and $\lambda_v = x_v^2$ for all $v \in V(G)$. Then
$$
I(G, s\ee - \xx \oplus t) = (s-t)^{2|V(G)|} I_{\boldsymbol{\lambda}}(G, -(s-t)^{-2}).
$$
By Theorem \ref{engstrom} the polynomial $I_{\boldsymbol{\lambda}}(G, s)$ is real-rooted. Clearly all roots are negative which implies $I_{\boldsymbol{\lambda}}(G, -s^{-2})$ is real-rooted. Hence the univariate polynomial $s \mapsto I(G, s\ee - \xx \oplus t)$ is real-rooted which shows that $I(G, \xx \oplus t)$ is hyperbolic with respect to $\ee$.

\end{proof} \noindent
An induced clique $K$ in $G$ is called a \textit{simplicial clique} if for all $u \in K$ the induced subgraph $N[u] \cap (G \setminus K)$ of $G \setminus K$ is a clique. In other words the neighbourhood of each $u \in K$ is a disjoint union of two induced cliques in $G$. Furthermore, a graph $G$ is said to be \textit{simplicial} if $G$ is claw-free and contains a simplicial clique.
 
In this section we prove Conjecture \ref{glc} for the polynomial $I(G, \xx,t)$ when $G$ is simplicial. The proof unfolds in a parallel manner to Theorem \ref{matchconespec} by considering a different kind of path tree. Before the results can be stated we must outline the necessary definitions from \cite{LR}.

A connected graph $G$ is a \textit{block graph} if each $2$-connected component is a clique. 
Given a simplicial graph $G$ with a simplicial clique $K$ we recursively define a block graph $T^{\boxtimes}(G,K)$ called the \textit{clique tree} associated to $G$ and rooted at $K$ (see Figure \ref{clfig}). 

We begin by adding $K$ to $T^{\boxtimes}(G,K)$. Let $K_u = N[u] \setminus K$ for each $u \in K$.
Attach the disjoint union $\bigsqcup_{u \in K} K_u$ of cliques to $T^{\boxtimes}(G,K)$ by connecting $u \in K$ to every $v \in K_u$. Finally recursively attach $T^{\boxtimes}(G\setminus K,K_u)$ to the clique $K_u$ in $T^{\boxtimes}(G,K)$ for every $u \in K$. Note that the recursion is made well-defined by the following lemma. 
\begin{lemma}[Chudnovsky-Seymour \cite{CS}]
Let $G$ be a clawfree graph and let $K$ be a simplicial clique in $G$. Then $N[u] \setminus K$ is a simplicial clique in $G \setminus K$ for all $u \in K$.
\end{lemma} \noindent

\begin{figure} 
    \centering
\begin{tabular}{ cc }
\begin{tikzpicture}
% create the node
\node[draw=white,minimum size=3cm,regular polygon,regular polygon sides=6] (p) {};

\node[label={\scriptsize a}, draw=black,fill,circle,inner sep=0pt,minimum size=3pt] (a) at (p.corner 1) {};

\node[label={\scriptsize b}, draw=black,fill,circle,inner sep=0pt,minimum size=3pt] (b) at (p.corner 2) {};

\node[label=left:{\scriptsize c}, draw=black,fill,circle,inner sep=0pt,minimum size=3pt] (c) at (p.corner 3) {};

\node[label=below:{\scriptsize d}, draw=black,fill,circle,inner sep=0pt,minimum size=3pt] (d) at (p.corner 4) {};

\node[label=below:{\scriptsize e}, draw=black,fill,circle,inner sep=0pt,minimum size=3pt] (e) at (p.corner 5) {};

\node[label=right:{\scriptsize f}, draw=black,fill,circle,inner sep=0pt,minimum size=3pt] (f) at (p.corner 6) {};

\draw[red] (a) -- (c);
\draw[red] (a) -- (b);
\draw[red] (b) -- (c);
\draw (c) -- (d);
\draw (d) -- (e);
\draw (e) -- (f);
\draw (a) -- (f);

\draw (a) -- (e);
\draw (b) -- (d);
\draw (b) -- (f);
\draw (c) -- (e);
\draw (d) -- (f);

\end{tikzpicture} 
\\ $G$ \\ \\
\begin{tikzpicture}
% create the node
\node[draw=red,minimum size=2cm,regular polygon,regular polygon sides=3] (t1) {};
\node[label=left:{\scriptsize b}, draw=black,fill,circle,inner sep=0pt,minimum size=3pt] (s1) at (t1.corner 1) {};
\node[label=above left:{\scriptsize c}, draw=black,fill,circle,inner sep=0pt,minimum size=3pt] (b1) at (t1.corner 2) {};
\node[label=above right:{\scriptsize a}, draw=black,fill,circle,inner sep=0pt,minimum size=3pt] (a1) at (t1.corner 3) {};

\node[draw=black,minimum size=2cm,regular polygon, rotate=180, regular polygon sides=3] (t2) at (0,2) {};
\node[label=right:{\scriptsize d}, draw=black,fill,circle,inner sep=0pt,minimum size=3pt] (d3) at (t2.corner 2) {};
\node[label=left:{\scriptsize f}, draw=black,fill,circle,inner sep=0pt,minimum size=3pt] (f3) at (t2.corner 3) {};

\node[draw=black,minimum size=2cm,regular polygon, rotate=180, regular polygon sides=3] (t3) at (1.75,-1) {};
\node[label=right:{\scriptsize e},draw=black,fill,circle,inner sep=0pt,minimum size=3pt] (e2) at (t3.corner 1) {};
\node[label=above:{\scriptsize f}, draw=black,fill,circle,inner sep=0pt,minimum size=3pt] (f2) at (t3.corner 2) {};

\node[draw=black,minimum size=2cm,regular polygon, rotate=180, regular polygon sides=3] (t4) at (-1.75,-1) {};
\node[label=left:{\scriptsize e},draw=black,fill,circle,inner sep=0pt,minimum size=3pt] (e1) at (t4.corner 1) {};
\node[label=above:{\scriptsize d}, draw=black,fill,circle,inner sep=0pt,minimum size=3pt] (d1) at (t4.corner 3) {};

\node[label=above:{\scriptsize f},draw=black,fill,circle,inner sep=0pt,minimum size=3pt] (fd) at ([yshift = 0.05pt, xshift=-45pt]d1) {};
\draw (d1) -- (fd);

\node[label=left:{\scriptsize f},draw=black,fill,circle,inner sep=0pt,minimum size=3pt] (fe) at (-2.4, -3.2) {};
\draw (e1) -- (fe);

\node[label=right:{\scriptsize d},draw=black,fill,circle,inner sep=0pt,minimum size=3pt] (de) at (2.4, -3.2) {};
\draw (e2) -- (de);

\node[label=above:{\scriptsize d},draw=black,fill,circle,inner sep=0pt,minimum size=3pt] (df) at ([yshift = 0.05pt, xshift=45pt]f2) {};
\draw (f2) -- (df);

\node[label=right:{\scriptsize e},draw=black,fill,circle,inner sep=0pt,minimum size=3pt] (ed) at (1.7, 4) {};
\draw (d3) -- (ed);

\node[label=left:{\scriptsize e},draw=black,fill,circle,inner sep=0pt,minimum size=3pt] (ef) at (-1.7, 4) {};
\draw (f3) -- (ef);

\end{tikzpicture}
\\ %\hline
$T^\boxtimes(G, \{a,b,c\})$
\end{tabular}
\caption{A simplicial graph $G$ and its associated relabelled clique tree $T^{\boxtimes}(G,K)$ rooted at $K = \{a,b,c \}$ (highlighted in red).}
    \label{clfig}
\end{figure}
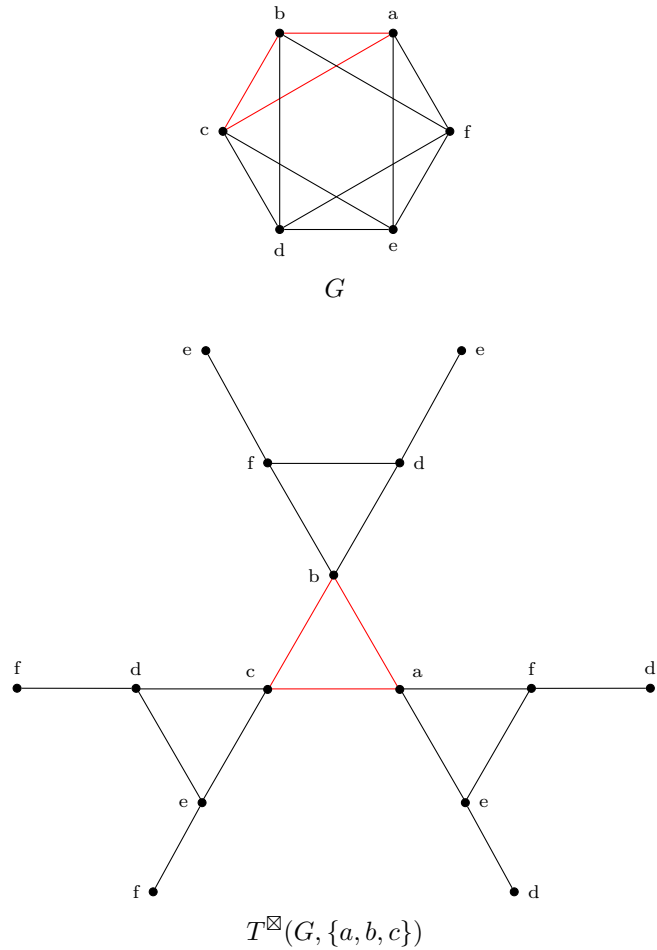

It is well-known that a graph is the line graph of a tree if and only if it is a claw-free block graph \cite[Thm 8.5]{Har}. In \cite{LR} it was demonstrated that the block graph $T^{\boxtimes}(G,K)$ is the line graph of a certain \textit{induced path tree} $T^{\angle}(G,K)$. Its precise definition is not important to us, but we remark that it is a subtree of the usual path tree defined in Section \ref{sec::matching} that avoids traversed neighbours. This enables us to find a definite determinantal representation of $I(T^{\boxtimes}(G,K), \xx \oplus t)$ via Lemma \ref{treedetrep}.
The second important fact is that $I(G, \xx)$ divides $I(T^{\boxtimes}(G,K), \xx)$ where $T^{\boxtimes}(G,K)$ is relabelled according to the natural graph homomorphism $\phi_K: T^{\boxtimes}(G,K) \to G$. 
Hence using the recursion provided by the simplicial structure of $G$ we have almost all the ingredients to finish the proof of Conjecture \ref{glc} for $I(G, \xx \oplus t)$.   

\begin{lemma}[Leake-Ryder \cite{LR})] \label{linecltree} \noindent \newline
For any simplicial graph $G$, and any simplicial clique $K \leq G$, we have 
$$
L(T^{\angle}(G,K)) \cong T^{\boxtimes}(G,K).
$$
\end{lemma} \noindent \newline
The following theorem is a generalization of Godsil's divisibility theorem for matching polynomials. It can be proved in a similar manner by induction using the recursive structure of simplicial graphs and removing cliques instead of vertices. For the proof to go through in the homogeneous setting we must replace the usual recursion by
$$
I(G, \xx \oplus t) = t^{2|K|} I(G \setminus K, \xx \oplus t) - \sum_{v \in K} t^{2|N(v)|} x_v^2 I(G \setminus N[v], \xx \oplus t).
$$
\begin{theorem} [Leake-Ryder \cite{LR})] \label{indepdiv}
Let $K$ be a simplicial clique of the simplicial graph $G$. Then
$$
\frac{I(G, \xx \oplus t)}{I(G \setminus K, \xx \oplus t)} = \frac{I(T^{\boxtimes}(G,K), \xx \oplus t)}{I(T^{\boxtimes}(G,K) \setminus K, \xx \oplus t)},
$$
where $T^{\boxtimes}(G,K)$ is relabelled according to the natural graph homomorphism $\phi_K: T^{\boxtimes}(G,K) \to G$. Moreover $I(G, \xx \oplus t)$ divides $I(T^{\boxtimes}(G,K), \xx \oplus t)$.
\end{theorem} \noindent
The following lemma ensures the hyperbolicity cones behave well under vertex deletion.
\begin{lemma} \label{hypconedel}
Let $v \in V(G)$. Then $\Lambda_{++}(I(G, \xx \oplus t)) \subseteq \Lambda_{++}(I(G \setminus v, \xx \oplus t))$.
\end{lemma}
\begin{proof}
Let $\xx \oplus t \in \RR^{V(G)} \times \RR$ and $\ee = (0, \dots, 0,1)$. By Lemma \ref{indephyp} the polynomials $s\mapsto I(G, s\ee - \xx \oplus t)$ and $s \mapsto I(G\setminus v, s\ee - \xx \oplus t)$ are both real-rooted. Denote their roots by $\alpha_1, \dots, \alpha_{2n}$ and $\beta_1, \dots, \beta_{2n-2}$ respectively where $n = |V(G)|$. We claim that
$$
\min_i \alpha_i \leq \min_{i} \beta_i \leq \max_{i} \beta_i \leq \max_i \alpha_i 
$$
by induction on the number of vertices of $G$. Indeed the claim is vacuously true if $|V(G)| = 1$.  
Suppose therefore $|V(G)| > 1$. If $G$ is not connected, then $G = G_1 \sqcup G_2$ for some non-empty graphs $G_1,G_2$. Without loss assume $v \in G_1$. Then $G \setminus v = (G_1\setminus v) \sqcup G_2$. By induction the claim holds for the pair $G_1$ and $G_1 \setminus v$. This implies the claim for $G$ and $G \setminus v$ since $I(G, \xx \oplus t)$ is multiplicative with respect to disjoint union.
We may therefore assume $G$ is connected. Thus $G \setminus N[v]$ is of strictly smaller size than $G \setminus v$.
We have
\begin{align} \label{indeprec}
\displaystyle I(G, \xx \oplus t) = t^2 I(G \setminus v, \xx \oplus t) - x_v^2 t^{2|N(v)|} I(G \setminus N[v], \xx \oplus t).
\end{align}
By induction, the maximal root $\gamma$ of $I(G \setminus N[v], s\ee - \xx \oplus t)$ is less than the maximal root $\beta$ of $I(G \setminus v, s\ee - \xx \oplus t)$. Since $I(G \setminus N[v], s\ee - \xx \oplus t)$ is an even degree polynomial with positive leading coefficient we have that $I(G \setminus N[v], s \ee - \xx \oplus t) \geq 0$ for all $s \geq \gamma$. By (\ref{indeprec}) this implies that $I(G, \beta \ee - \xx \oplus t) \leq 0$. Hence $\max_{i} \beta_i \leq \max_{i} \alpha_i$ since $I(G, s \ee - \xx \oplus t) \to \infty$ as $s \to \infty$. Since each of the terms involved in the polynomials $I(G, s \ee - \xx \oplus t)$ and $I(G \setminus v, s \ee - \xx \oplus t)$ have even degree in $s-t$, their respective roots are symmetric about $s = t$. Hence $\min_i \alpha_i \leq \min_i \beta_i$ proving the claim.
Finally if $\xx_0 \oplus t_0 \in \Lambda_{++}(I(G, \xx \oplus t))$, then $\min_i \alpha_i > 0$ so by the claim $\min_i \beta_i > 0$ showing that $\xx_0 \oplus t_0 \in \Lambda_{++}(I(G \setminus v, \xx \oplus t))$. This proves the lemma.

\end{proof}
\begin{remark}
Since 
$$
I(G, \xx \oplus t) \biggr |_{x_v = 0} = t^{2} I(G\setminus v, \xx \oplus t),
$$
we see by Lemma \ref{hypconedel} that setting vertex variables equal to zero relaxes the hyperbolicity cone.
\end{remark}
\begin{theorem} \label{indepspec}
If $G$ is a simplicial graph, then the hyperbolicity cone of $I(G, \xx \oplus t)$ is spectrahedral.  
\end{theorem}
\begin{proof}
Let $K$ be a simplicial clique of $G$. Arguing by induction as in Theorem \ref{matchconespec}, using the clique tree $T^{\boxtimes}(G,K)$ instead of the path tree $T(G,u)$, and invoking Theorem \ref{indepdiv} we get a factorization 
\begin{align} \label{clfac}
\displaystyle q_{G,K}(\xx \oplus t) = q_{G \setminus K, K_v}(\xx \oplus t) \prod_{w \in K \setminus v} q_{G\setminus K, K_w}(\xx \oplus t) I(G \setminus K, \xx \oplus t),
\end{align} \noindent
where $v \in K$ is fixed, $K_w = N[w]\setminus K$ and 
\begin{align*}
q_{G, K}(\xx \oplus t)I(G, \xx \oplus t) &= I(T^{\boxtimes}(G,K), \xx \oplus t), \\
q_{G \setminus K, K_w}(\xx \oplus t) I(G \setminus K, \xx \oplus t) &= I(T^{\boxtimes}(G \setminus K,K_w), \xx \oplus t)
\end{align*} \noindent
for $w \in K$.
Repeated application of Lemma \ref{hypconedel} gives
$$
\Lambda_{++}(I(G, \xx \oplus t)) \subseteq \Lambda_{++}(I(G \setminus K, \xx \oplus t)).
$$
By the factorization (\ref{clfac}) and induction we hence get the desired cone inclusion 
$$
\Lambda_{++}(I(G, \xx \oplus t)) \subseteq \Lambda_{++}(q_{G,K}(\xx \oplus t)).
$$
Since $L(T^{\angle}(G,K)) \cong T^{\boxtimes}(G,K)$ by Lemma \ref{linecltree} we see that
$$
I(T^{\boxtimes}(G,K), \xx \oplus t) = \mu(T^{\angle}(G,K), t\mathbf{1} \oplus \xx).
$$
Hence $I(T^{\boxtimes}(G,K), \xx \oplus t)$ has a definite determinantal representation by Lemma \ref{treedetrep} proving the theorem.

\end{proof}

\section{Convolutions} \label{sec::conv}
If $G$ is a simple undirected graph with adjacency matrix $A = (a_{ij})$, then we may associate a signing $\ss = (s_{ij}) \in \{\pm 1 \}^{E(G)}$ to its edges. The symmetric adjacency matrix $A^{\ss} = (a_{ij}^s)$ of the resulting graph is given by $a_{ij}^{\ss} = s_{ij} a_{ij}$ for $ij \in E(G)$ and $a_{ij}^{\ss} = 0$ otherwise.
Godsil and Gutman \cite{GG2} proved that
\begin{align} \label{godsilgutman}
\mathop{\mathbb{E}}_{\ss \in \{ \pm 1 \}^{E(G)}} \det \left ( tI - A^{\ss} \right ) = \mu(G,t).
\end{align}
In other words, the expected characteristic polynomial of
an independent random signing of the adjacency matrix of a graph is equal to its matching polynomial. Therefore the expected characteristic polynomial is real-rooted. This was one of the facts used by Marcus, Spielman and Srivastava \cite{MSS1} in proving that there exist infinite families of regular bipartite Ramanujan graphs. Since then, several other families of characteristic polynomials have been identified with real-rooted expectation (see e.g \cite{MSS3},\cite{HDS}). Such families go under the name \textit{interlacing families}, based on the fact that there exists a common root interlacing polynomial if and only if every convex combination of the family is real-rooted. The method of interlacing families have been successfully applied to other contexts, in particular to the affirmative resolution of the Kadison-Singer problem \cite{MSS2}. 

In this section we define a convolution of multivariate determinant polynomials and show that it is hyperbolic as a direct consequence of a more general theorem by Br\"and\'en \cite{B3}. In particular this convolution can be viewed as a generalization of the fact that the expectation in (\ref{godsilgutman}) is real-rooted. Namely, we show that the expected characteristic polynomial over any finite set of independent random edge weightings is real-rooted barring certain adjustments to the weights of the loop edges.  

Recall that every symmetric matrix may be identified with the adjacency matrix of an undirected weighted graph (with loops). 
\begin{definition}
Let $W \subseteq \RR$ be a finite set and let $A = (a_{ij})_{i,j = 1}^n$ be a real symmetric matrix. Define a \textit{weighting} $\ww = (w_{ij})_{i < j} \in W^{\binom{n}{2}}$ of $A$ to be a symmetric matrix $A^{\ww} = (a_{ij}^{\ww})$ given by 
$$
a_{ij}^{\ww} = \begin{cases} w_{ij} a_{ij} & \text{ if } i < j \\ \sum_{k = 1}^i a_{ik} + \sum_{k = i+1}^n w_{ik}^2 a_{ik} & \text{ if } i = j \end{cases}.
$$
\end{definition}
\begin{definition}
Let $X = (x_{ij})_{i,j=1}^n$ and $Y = (y_{ij})_{i,j=1}^n$ be symmetric matrices in variables $\xx = (x_{ij})_{i \leq j}$ and $\yy = (y_{ij})_{i \leq j}$ respectively. Let $W \subseteq \RR$ be a finite set. We define the convolution
\begin{align*}
\displaystyle \det(X) \conv_{W} \det(Y) =  \mathop{\mathbb{E}}_{\ww_1, \ww_2 \in W^{\binom{n}{2}}} \det( X^{\ww_1} + Y^{\ww_2}) \in \mathbb{R}[\xx, \yy].
\end{align*}
\end{definition} \noindent
We have the following general fact about hyperbolic polynomials. 
\begin{theorem}[Br\"and\'en \cite{B3}]
\label{compatiblefamilies}
Let $h(\xx)$ be a hyperbolic polynomial with respect to $\ee \in \mathbb{R}^n$, let $V_1, \dots, V_m$ be finite sets of vectors of rank at most one in $\Lambda_+$. For $\mathbf{V} = (\vv_1, \dots, \vv_m) \in V_1 \times \cdots \times V_m$, let 
\begin{align*}
\displaystyle g(\mathbf{V};t) = h(t\ee + \uu -\alpha_1 \vv_1 - \cdots - \alpha_m \vv_m)
\end{align*} \noindent
where $\uu \in \mathbb{R}^n$ and $(\alpha_1, \dots, \alpha_m) \in \mathbb{R}^{m}$. Then $\displaystyle \mathop{\mathbb{E}}_{\mathbf{V} \in V_1 \times \cdots \times V_m} \thinspace g(\mathbf{V};t)$ is real-rooted.
\end{theorem} \noindent
\begin{proposition} \label{convhyp}
Let $W \subseteq \mathbb{R}$ be a finite subset. Then $\det(X) \conv_W \det(Y)$ is hyperbolic with respect to $\ee = I \oplus \mathbf{0}$ where $I$ denotes the identity matrix.
\end{proposition}
\begin{proof}
Let $h(X \oplus Y) = \det(X) \conv_{W} \det(Y)$.
We note that $h(\ee) = 1 \neq 0$.
Let $\delta_1, \dots, \delta_n$ denote the standard basis of $\mathbb{R}^n$. Put 
$$
V_{ij} = \{ \vv_{ijw} : w \in W \}
$$ 
where $\vv_{ijw} =  (\delta_{i} + w \delta_j) (\delta_{i} + w \delta_j)^{T}$ for $i < j$ and $w \in W$. Note that $\vv_{ijw}$ is a rank one matrix belonging to the hyperbolicity cone of positive semidefinite matrices (with non-zero eigenvalue $w^2+1$). Letting $\uu = \mathbf{0}$ and $\alpha_{ij}^X = x_{ij}, \alpha_{ij}^Y = y_{ij}$ for $i < j$ we see that
\begin{align*}
\displaystyle h(t \ee - X \oplus Y) &= \mathop{\mathbb{E}}_{\ww_1, \ww_2} \det(tI - X^{\ww_1} - Y^{\ww_2}) \\ &= \mathop{\mathbb{E}}_{\substack{\vv_{ijw_1}, \vv_{ijw_2} \in V_{ij} \\ i < j}} \det \left (tI + \uu - \sum_{i < j} (\alpha_{ij}^X \vv_{ijw_1} +  \alpha_{ij}^Y \vv_{ijw_2}) \right ),
\end{align*} \noindent
where the right hand side is a real-rooted polynomial in $t$ by Theorem \ref{compatiblefamilies}. Hence $\det(X) \conv_W \det(Y)$ is hyperbolic with respect to $\ee$.

\end{proof}
\begin{remark} \label{diagrmk}
Taking $W = \{ \pm 1 \}$ we see that the diagonal adjustment in the weighting is constant. Therefore setting $\uu = \text{diag}(d_1, \dots, d_n)$ where $d_i = \sum_{j \neq i} (x_{ij} + y_{ij})$ in the proof of Corollary \ref{convhyp}, we get that
\begin{align} \label{sgnconv}
\displaystyle \mathop{\mathbb{E}}_{\ss_1, \ss_2} \det(X^{\ss_1} + Y^{\ss_2})
\end{align} \noindent
is hyperbolic, where the expectation is taken over independent random signings of the matrices $X$ and $Y$ as in (\ref{godsilgutman}) without diagonal adjustment. This shows in particular that the expectation in (\ref{godsilgutman}) is real-rooted.
\end{remark}
\begin{corollary}
Let $W \subseteq \RR$ be a finite subset and $A$ a real symmetric $n \times n$ matrix. Then
\begin{align*} 
\displaystyle \mathop{\mathbb{E}}_{\ww \in W^{\binom{n}{2}}} \det(tI - A^{\ww})
\end{align*} \noindent
is real-rooted.
\end{corollary}
\begin{proof}
By Corollary \ref{convhyp} the polynomial $\det(Y) \conv_W \det(X)$ is hyperbolic, so in particular $\displaystyle t \mapsto \mathop{\mathbb{E}}_{\ww} \det(tI - A^{\ww})$ is real-rooted with $X = \mathbf{0}$ and $Y = A$.

\end{proof} \noindent
Next we see that the convolution (\ref{sgnconv}) over independent random signings can be realized as a convolution of multivariate matching polynomials. The proof is similar to that of the univariate identity (\ref{godsilgutman}) (cf \cite{GG2}). Let $G_X$ and $G_Y$ denote the weighted graphs corresponding to the symmetric matrices $X$ and $Y$. 
\begin{proposition}
\label{boxplus} 
Let $X = (x_{ij})_{i,j = 1}^n$ and $Y = (y_{ij})_{i,j = 1}^n$ be symmetric matrices in variables $\xx = (x_{ij})_{i \leq j}$ and $\yy = (y_{ij})_{i \leq j}$. Then
\begin{align*}
\displaystyle &\mathop{\mathbb{E}}_{\ss^{(1)}, \ss^{(2)}} \det(X^{\ss^{(1)}} + Y^{\ss^{(2)}}) \\ &= \sum_{S \subseteq [n]} (-1)^{|S|/2} \prod_{i \not \in S} (x_{ii} + y_{ii}) \sum_{S_1 \sqcup S_2 = S} \mu(G_X[S_1], \mathbf{0} \oplus \mathbf{x}) \mu(G_Y[S_2], \mathbf{0} \oplus \mathbf{y})
\end{align*} \noindent
where the expectation is taken over independent random signings as in (\ref{godsilgutman}). 
\end{proposition}  
\begin{proof}
Expanding the convolution from the definition of the determinant we have
\begin{align*}
\displaystyle  &\mathop{\mathbb{E}}_{\ss^{(1)}, \ss^{(2)}} \det( X^{\ss^{(1)}} + Y^{\ss^{(2)}}) \\  &= \mathop{\mathbb{E}}_{\ss^{(1)}, \ss^{(2)}} \sum_{\sigma \in \mathfrak{S}_{n}} \text{sgn}(\sigma) \prod_{i = 1}^{n} \left (X^{\ss^{(1)}} + Y^{\ss^{(2)}} \right )_{i \sigma(i)} \\ & = \mathop{\mathbb{E}}_{\ss^{(1)}, \ss^{(2)}} \sum_{S \subseteq [n]} \prod_{i \not \in S} (x_{ii} + y_{ii}) \sum_{\substack{\sigma \in \text{Sym}(S) \\ \sigma(j) \neq j \thickspace \forall j \in S}} \text{sgn}(\sigma) \prod_{j \in  S} \left (  s_{j \sigma(j)}^{(1)}  x_{j \sigma(j)} + s_{j \sigma(j)}^{(2)} y_{j \sigma(j)} \right ) \\ &= 
\sum_{S \subseteq [n]} \prod_{i \not \in S} (x_{ii} + y_{ii}) \sum_{\substack{\sigma \in \text{Sym}(S) \\ \sigma(j) \neq j \thickspace \forall j \in S}} \text{sgn}(\sigma) \sum_{S_1 \sqcup S_2 = S} \mathop{\mathbb{E}}_{\ss^{(1)}} \prod_{j \in  S_1} s_{j \sigma(j)}^{(1)} x_{j \sigma(j)} \mathop{\mathbb{E}}_{\ss^{(2)}} \prod_{j \in  S_2} s_{j \sigma(j) }^{(2)} y_{j \sigma(j)}
\end{align*} \noindent
Note the following regarding the random variables $s_{ij}^{(k)}$, $k = 1,2$:
\begin{enumerate}
\item $s_{ij}^{(k)}$ appears with power at most two in each of the products.
\item The random variables $s_{ij}^{(k)}$ are independent.
\item $\mathop{\mathbb{E}} s_{ij}^{(k)}  = 0$.
\item $\mathop{\mathbb{E}} (s_{ij}^{(k)})^2 = 1$.
\end{enumerate} \noindent
As a consequence, permutations with the following characteristics may be eliminated since they produce factors $s_{ij}^{(k)}$ of power one making the term vanish:
\begin{enumerate}
\item $\sigma \in \mathfrak{S}_n$ having no factorization $\sigma = \sigma_1 \sigma_2$ for $\sigma_i \in \text{Sym}(S_i)$, $i = 1,2$.
\item $\sigma \in \mathfrak{S}_n$ such that $\sigma$ is not a complete product of disjoint transpositions. 
\end{enumerate} \noindent
This leaves us with products of fixed-point-free involutions in $\text{Sym}(S_1)$ and $\text{Sym}(S_2)$. Thus the non-vanishing terms are those corresponding to perfect matchings on $G_X[S_1]$ and $G_Y[S_2]$. Hence 
\begin{align*}
\displaystyle \mathop{\mathbb{E}}_{\ss^{(1)}, \ss^{(2)}} \det( X^{\ss^{(1)}} + Y^{\ss^{(2)}}) = \sum_{S \subseteq [n]} \prod_{i \not \in S} (x_{ii} + y_{ii}) \sum_{S_1 \sqcup S_2 = S} P_1(\xx) P_2(\yy)
\end{align*} \noindent
where
\begin{align*}
\displaystyle P_1(\xx) &= \sum_{\substack{ \sigma_1 \in \text{Sym}(S_1) \\ \sigma_1(j) \neq j \thickspace \forall j \in S_1}} \text{sgn}(\sigma_1) \mathop{\mathbb{E}}_{\ss^{(1)}} \prod_{i \in S_1} s_{i \sigma_1(i) }^{(1)} x_{i \sigma_1(i)} \\ &= \sum_{ \substack{ M \in \mathcal{M}(G_X[S_1]) \\ M \text{ perfect }}} (-1)^{|S_1|/2} \prod_{ij \in M} \mathop{\mathbb{E}}_{\ss^{(1)}} (s_{ i j }^{(1)})^2 x_{i j}^2 \\ &= (-1)^{|S_1|/2} \sum_{\substack{ M \in \mathcal{M}(G_X[S_1]) \\ M \text{ perfect }}} \prod_{ij \in M} x_{ij}^2 \\ &= (-1)^{|S_1|/2} \mu(G_X[S_1], \mathbf{0} \oplus \mathbf{x})
\end{align*} \noindent
and similarly for $P_2(\yy)$.

\end{proof}  \noindent
\begin{remark}
The expression in Proposition \ref{boxplus} may also be written
\begin{align*}
\displaystyle \displaystyle \mathop{\mathbb{E}}_{\ss^{(1)}, \ss^{(2)}} \det( X^{\ss^{(1)}} + Y^{\ss^{(2)}}) = \sum_{M \in \mathcal{M}(K_n)} (-1)^{|M|} \prod_{i \not \in V(M)} (x_{ii} + y_{ii}) \prod_{jk \in M} (x_{jk}^2 + y_{jk}^2).
\end{align*} \noindent
%where $K_n$ denotes the complete graph on $n$ vertices.
\end{remark}
\begin{example} \noindent
\begin{enumerate}
\item Let $A$ be the adjacency matrix of a simple undirected graph $G$. Under the specialization $X = tI$ and $Y = -A$ in Proposition \ref{boxplus} we recover the identity (\ref{godsilgutman}) of Godsil and Gutman.
\item Let $A$ and $B$ both be adjacency matrices of the complete graph $K_n$. It is well-known (see e.g. \cite{God}) that the number of perfect matchings in $K_n$ is given by $(n-1)!!$ if $n$ is even and $0$ otherwise, where $(n)!! = n(n-2)(n-4)\cdots$. 
By Proposition \ref{boxplus} and a simple calculation it follows that
\begin{align*}
\displaystyle \mathop{\mathbb{E}}_{\ss^{(1)}, \ss^{(2)}} \det( tI + A^{\ss^{(1)}} + B^{\ss^{(2)}}) &= \sum_{k = 0}^{\lfloor n/2 \rfloor} t^{n-2k} (-1)^k \binom{n}{2k} \sum_{i+j = k} \binom{2k}{2i} (2i-1)!!(2j-1)!! \\ &= \sum_{k = 0}^{\lfloor n/2 \rfloor} t^{n-2k} (-1)^k \binom{n}{2k} (2k-1)!! \left ( \frac{3}{2} \right )^k \\ &= t^n \mu_{\frac{3}{2} \mathbf{1}}(K_n, t^{-1}\mathbf{1}).
\end{align*} \noindent
\end{enumerate}
\end{example} \noindent
\section*{Final remarks}
In Theorem \ref{indepspec} we proved Conjecture \ref{glc} for $I(G, \xx \oplus t)$ whenever $G$ is a simplicial graph. An extension of the divisibility relation in Theorem \ref{indepdiv} to all claw-free graphs would immediately extend Theorem \ref{indepspec} to all claw-free graphs. 

An interesting extension of this work would be to study a family of stable graph polynomials introduced by Wagner \cite{Wag} in a general effort to prove Heilmann-Lieb type theorems. Let $G = (V,E)$ be a graph. For $H \subseteq E$, let $\deg_H: V \to \mathbb{N}$ denote the degree function of the subgraph $(V,H)$. Furthermore let 
\begin{align*}
\displaystyle \uu^{(v)} = (u_0^{(v)}, u_1^{(v)}, \dots, u_d^{(v)})
\end{align*} \noindent
denote a sequence of \textit{activities} at each vertex $v \in V$ where $d = \deg_G(v)$. 
Define the polynomial
\begin{align*}
\displaystyle Z(G, \bl, \uu; \xx ) = \sum_{H \subseteq E} (-1)^{|H|} \bl^H \uu_{\deg_H} \xx^{\deg_H}
\end{align*} \noindent
where $\bl = \{\lambda_e\}_{e \in  E}$ are edge weights and 
\begin{align*}
\displaystyle \bl^H = \prod_{e \in H} \lambda_e, \hspace{0.3cm} \uu_{\deg_H} = \prod_{v \in V} u_{\deg_H(v)}^{(v)}, \hspace{0.3cm} \xx^{\deg_H} = \prod_{v \in .V} x_v^{\deg_H(v)}.
\end{align*} \noindent
Wagner proves that $Z(G, \boldsymbol{\lambda}, \uu, \xx)$ is stable whenever $\lambda_e \geq 0$ for all $e \in E$ and the univariate \textit{key-polynomial} $K_v(z) = \sum_{j=0}^d \binom{d}{j} u_j^{(v)} z^j$ is real-rooted for all $v \in V$ (cf \cite[Thm 3.2]{Wag}). We note in particular that if $u_0^{(v)} = u_1^{(v)} = 1$, $u_k^{(v)} = 0$ for all $k > 1$ and $\thickspace v \in V$, then $Z(G, \bl, \uu; \xx) = \mu_{\bl}(G,\xx)$ where $\mu_{\bl}(G,\xx)$ is the weighted multivariate matching polynomial studied by Heilmann and Lieb \cite{HL}. An appropriate homogenization of $Z(G, \bl, \uu; \xx)$ could be defined as
$$
W(G, \mathbf{u}; \xx \oplus \ww) = \sum_{H \subseteq E} (-1)^{|H|} \uu_{\deg_H} \ww^{2H} \xx^{\deg_G - \deg_H}.
$$ 
Since $W(G, \mathbf{u}; \xx \oplus \ww) = \xx^{\deg_G} Z(G, \ww^2, \uu; \xx^{-1} )$ we see that $W(G, \mathbf{u}; \xx \oplus \ww)$ is hyperbolic with respect to $\ee = \mathbf{1} \oplus \mathbf{0}$ whenever $K_v(z)$ is real-rooted for all $v \in  V$. We also note the following edge and node recurrences for $e \in E$ and $v \in V$,
\begin{align*}
&W(G, \uu; \xx \oplus \ww ) \\ &= \xx^e W(G \setminus e, \uu; \xx \oplus \ww) - \ww^{2e} W(G \setminus e, \uu \ll e; \xx \oplus \ww) \\
&= \sum_{S \subseteq N(v)} (-1)^{|S|} u_{|S|}^{(v)} \ww^{2E(S,v)} x_v^{\deg_G(v) - |S|}\xx^{N(v) \setminus S}W(G \setminus v,\uu \ll S; \xx \oplus \ww)
\end{align*} \noindent
where $E(S,v) = \{ sv \in E : s \in S \}$ and $\displaystyle (\uu \ll S)^{(v)} = \begin{cases} (u_1^{(v)}, \dots, u_d^{(v)}), & v \in S \\ \uu^{(v)}, & v \not \in S \end{cases}$

Although it is not clear in general how to find a definite determinantal representation of $W(G, \mathbf{u}; \xx \oplus \ww)$, it may be possible to consider special form activity vectors and obtain a reduction by constructing divisibility relations in the spirit of Lemma \ref{matchquot} and Theorem \ref{indepdiv}. This may also be of independent interest for studying root bounds of their univariate specializations. \noindent \newline \newline
\textbf{Acknowledgements.} The author is grateful to Petter Br\"and\'en.

\end{document}